\documentclass[options]{article}
\usepackage{amsmath}
\usepackage{amsfonts}
\usepackage{amssymb}
\usepackage{amsthm}
\usepackage{esint}
\usepackage{mathtools}
\usepackage{titling}
\usepackage{authblk}
\usepackage{commath}
\usepackage{xcolor}
\usepackage{graphicx}
\usepackage[toc,title,page]{appendix}
\usepackage{subcaption}
\usepackage{bbm}
\usepackage{hyperref}
\usepackage{color,soul}
\usepackage{bm}
\usepackage{mathrsfs}

\usepackage{enumitem}
\newtheorem*{openpb}{Open Problem}
\newtheorem{remark}{Remark}[section]
\newtheorem{theorem}[remark]{Theorem}
\newtheorem{corollary}[remark]{Corollary}
\newtheorem{lemma}[remark]{Lemma}
\usepackage[margin=1in]{geometry}

\newtheorem{proposition}[remark]{Proposition}
\newtheorem{example}[remark]{Example}
\numberwithin{equation}{section}

\title{On a Class of Nonlocal Obstacle Type Problems \\Related to the Distributional Riesz Fractional Derivative}
\author{Catharine W.K. Lo\thanks{CMAFcIO -- Departamento de Matem\'atica, Faculdade de Ci\^encias, Universidade de Lisboa P-1749-016 Lisboa, Portugal\\ Email address: cwklo@fc.ul.pt} \, and Jos\'e Francisco Rodrigues\thanks{CMAFcIO -- Departamento de Matem\'atica, Faculdade de Ci\^encias, Universidade de Lisboa P-1749-016 Lisboa, Portugal\\ Email address: jfrodrigues@ciencias.ulisboa.pt}}
\date{}

\begin{document}
\maketitle

\begin{abstract}
    In this work, we consider the nonlocal obstacle problem with a given obstacle $\psi$ in a bounded Lipschitz domain $\Omega$ in $\mathbb{R}^{d}$, such that $\mathbb{K}_\psi^s=\{v\in H^s_0(\Omega):v\geq\psi \text{ a.e. in }\Omega\}\neq\emptyset$, given by \[u\in\mathbb{K}_\psi^s:\quad\langle\mathcal{L}_au,v-u\rangle\geq\langle F,v-u\rangle\quad\forall v\in\mathbb{K}^s_\psi,\] for $F$ in $H^{-s}(\Omega)$, the dual space of the fractional Sobolev space $H^s_0(\Omega)$, $0<s<1$. The nonlocal operator $\mathcal{L}_a:H^s_0(\Omega)\to H^{-s}(\Omega)$ is defined with a measurable, bounded, strictly positive singular kernel $a(x,y):\mathbb{R}^d\times\mathbb{R}^d\to[0,\infty)$, by     the bilinear form \[\langle\mathcal{L}_au,v\rangle=P.V.\int_{\mathbb{R}^d}\int_{\mathbb{R}^d} \tilde{v}(x)(\tilde{u}(x)-\tilde{u}(y))a(x,y) \,dy\,dx=\mathcal{E}_a(u,v),\] which is a (not necessarily symmetric) Dirichlet form, where $\tilde{u},\tilde{v}$ are the zero extensions of $u$ and $v$ outside $\Omega$ respectively. Furthermore, we show that the fractional operator $\tilde{\mathcal{L}}_A=-D^s\cdot AD^s:H^s_0(\Omega)\to H^{-s}(\Omega)$ defined with the distributional Riesz fractional $D^s$ and with a measurable, bounded matrix $A(x)$ corresponds to a nonlocal integral operator $\mathcal{L}_{k_A}$ with a well defined integral singular kernel $a=k_A$. The corresponding $s$-fractional obstacle problem for $\tilde{\mathcal{L}}_A$ is shown to converge as $s\nearrow1$ to the obstacle problem in $H^1_0(\Omega)$ with the operator $-D\cdot AD$ given with the classical gradient $D$.
    
    We mainly consider obstacle-type problems involving the bilinear form $\mathcal{E}_a$ with one or two obstacles, as well as the N-membranes problem, thereby deriving several results, such as the weak maximum principle, comparison properties, approximation by bounded penalization, and also the Lewy-Stampacchia inequalities. This provides regularity of the solutions, including a global estimate in $L^\infty(\Omega)$, local H\"older regularity of the solutions when $a$ is symmetric, and local regularity in fractional Sobolev spaces $W^{2s,p}_{loc}(\Omega)$ and in $C^1(\Omega)$ when $\mathcal{L}_a=(-\Delta)^s$ corresponds to fractional $s$-Laplacian obstacle-type problems \[u\in\mathbb{K}^s_\psi(\Omega):\int_{\mathbb{R}^d}(D^su-\bm{f})\cdot D^s(v-u)\,dx\geq0\quad\forall v\in\mathbb{K}^s_\psi\text{ for }\bm{f}\in [L^2(\mathbb{R}^d)]^d.\]  These novel results are complemented with the extension of the Lewy-Stampacchia inequalities to the order dual of $H^s_0(\Omega)$ and some remarks on the associated $s$-capacity and the $s$-nonlocal obstacle problem for a general $\mathcal{L}_a$.

\end{abstract}

\section{Introduction}
Fractional problems with obstacle-type constraints were first considered by Silvestre as part of his thesis in 2005 in \cite{SilvestreThesis}, which was published in 2007 in \cite{SilvestrePaper}. Since then, many obstacle-type problems with various nonlocal operators have been considered, extensively for the fractional Laplacian (such as in \cite{CaffarelliSalsaSilvestre}, \cite{DanielliSalsa}, \cite{MusinaNazarovVarIneqSpectral}, \cite{MusinaNazarovSreenadhVarIneqFracLap} and \cite{RosOtonObsPb}), as well as for other nonlocal operators (see, for example, \cite{AbatangeloRosOton}, \cite{AntilRautenberg}, \cite{CaffarelliDeSilvaSavin}, \cite{DanielliPetrosyanPopObsPbMarkov}, \cite{SalsaSignorini} and \cite{ServadeiValdinoci2013RMILewyStampacchia}), which are mainly generalizations of the fractional Laplacian to nonlocal one-variable kernels $K(y)$ satisfying homogeneous and symmetry properties (see in particular, \cite{AbatangeloRosOton} and \cite{CaffarelliDeSilvaSavin}).

Recently, in a series of two interesting papers \cite{SS1} and \cite{SS2}, Shieh and Spector have considered a new class of fractional partial differential equations based on the distributional Riesz fractional derivatives. These fractional operators satisfy basic physical invariance  requirements, as observed by Silhavy \cite{Silhavy}, who developed a fractional vector calculus for such operators. Instead of using the well-known fractional Laplacian, their starting concept is the distributional Riesz fractional gradient of order $s\in(0,1)$, which will be called here the $s$-gradient $D^s$, for brevity: for $u\in L^p(\mathbb{R}^d)$, $p\in(1,\infty)$, we set
\begin{equation}\label{sGradDef}D^s_ju=\frac{\partial^s u}{\partial x_j^s}=\frac{\partial}{\partial x_j}(I_{1-s}*u),\quad 0<s<1,\quad j=1,\dots, d,\end{equation}
where $\frac{\partial}{\partial x_j}$ is taken in the distributional sense, for every $v\in C_c^\infty(\mathbb{R}^d)$, \[\left\langle\frac{\partial^s u}{\partial x_j^s},v\right\rangle=-\left\langle (I_{1-s}*u),\frac{\partial v}{\partial x_j}\right\rangle=-\int_{\mathbb{R}^d}\left(I_{1-s}*u\right)\frac{\partial v}{\partial x_j}\,dx,\] with $I_s$ denoting the Riesz potential of order $s$, $0<s<1$: \begin{equation}\label{DsDefRiesz}(I_s*u)(x)=c_{d,1-s}\int_{\mathbb{R}^d}\frac{u(y)}{|x-y|^{d-s}}\,dy,\quad\text{ with }c_{d,s}=2^s\pi^{-\frac{d}{2}}\frac{\Gamma\left(\frac{d+s+1}{2}\right)}{\Gamma\left(\frac{1-s}{2}\right)}.\end{equation} Conversely, by Theorem 1.12 of  \cite{SS1}, every $u\in C_c^\infty(\Omega)$ can be expressed as \begin{equation}\label{InverseFormulaDs}u=I_s*\sum_{j=1}^d\mathcal{R}_j\frac{\partial^s u}{\partial x_j^s},\end{equation} where $\mathcal{R}_j$ is the Riesz transform, which we recall, is given by 
\[\mathcal{R}_jf(x):=\frac{\Gamma\left(\frac{d+1}{2}\right)}{\pi^{(d+1)/2}}\lim_{\epsilon\to0}\int_{\{|y|>\epsilon\}}\frac{y_j}{|y|^{d+1}}f(x-y)\,dy,\quad j=1,\dots,d.\]  Thus, we can write the $s$-gradient ($D^s$) and the $s$-divergence ($D^s\cdot$) for sufficiently regular functions $u$ and vector $\phi$ (\cite{Comi1}, \cite{SS1}, \cite{SS2}, \cite{Silhavy}) in integral form, respectively, by \[D^su(x):=c_{d,s}\lim_{\epsilon\to0}\int_{\mathbb{R}^d}\frac{zu(x+z)}{|z|^{d+s+1}}\chi_\epsilon(0,z)\,dz=c_{d,s}\int_{\mathbb{R}^d}\frac{u(x)-u(y)}{|x-y|^{d+s}}\frac{x-y}{|x-y|}\,dy\] and \[D^s\cdot\phi(x):=c_{d,s}\lim_{\epsilon\to0}\int_{\mathbb{R}^d}\frac{z\cdot \phi(x+z)}{|z|^{d+s+1}}\chi_\epsilon(0,z)\,dz=c_{d,s}\int_{\mathbb{R}^d}\frac{\phi(x)-\phi(y)}{|x-y|^{d+s}}\cdot\frac{x-y}{|x-y|}\,dy,\] where $\chi_\epsilon(x,z)$ is the characteristic function of the set $\{(x,z):|z-x|>\epsilon\}$ for $\epsilon>0$.  As it was shown in \cite{SS1}, $D^s$ has nice properties for $u\in C_c^\infty(\mathbb{R}^d)$, namely it coincides with the fractional Laplacian as follows: \[(-\Delta)^su=-D^s\cdot D^su,\] where, for $0<s<1$, \[(-\Delta)^su(x)=c_{d,s}^2\lim_{\epsilon\to0}\int_{\mathbb{R}^d}\frac{u(x)-u(y)}{|x-y|^{d+2s}}\chi_\epsilon(x,y)\,dy=\frac{1}{2}c_{d,s}^2\int_{\mathbb{R}^d}\frac{u(x+y)+u(x-y)-2u(x)}{|y|^{d+2s}}\,dy. \]

Observe that for the $s$-gradient ($D^s$) and the $s$-divergence ($D^s\cdot$) we need to consider the Cauchy principal value (P.V.) in the first expressions, but not in the second ones. This is because for the second expressions, we can estimate the integrand, as in \cite{Comi1}, by separating the integrals into the parts $\{|y-x|\leq1\}$ and $\{|y-x|>1\}$. Then the first integral can be controlled by $\omega_d\, Lip(u)\int_0^1 r^{-s}\,dr$, while the second integral can be controlled by $2\omega_d\norm{u}_{L^\infty(\mathbb{R}^d)}\int_1^{+\infty} r^{-(1+s)}\,dr$,
where $Lip(u)$ is the Lipschitz constant for the function $u$ and $\omega_d$ is the spherical measure $\omega_d=\int_{\{|x|=1\}}d\sigma$. Therefore, the second expressions are well-defined for all Lipschitz functions $u$ with compact support, in particular for $u\in C_c^\infty(\mathbb{R}^d)$.


Similarly, the fractional Laplacian for smooth $u$, by Lemma 3.2 of \cite{HitchhikerGuide}, has two representations, with the second one being Lebesgue integrable by using a second order Taylor expansion. As a matter of fact, the $s$-divergence, $s$-gradient and $s$-Laplacian are linear operators from $C^\infty$ functions with compact support into $C^\infty$ functions that are rapidly decreasing at $\infty$ and are in $L^p(\mathbb{R}^d)$ for any $p\in[1,\infty]$.

Observe that for $u,v\in C_c^\infty(\mathbb{R}^d)$, by using the duality between $s$-divergence and $s$-gradient as in Lemma 2.5 of \cite{Comi1}, we have 
\begin{align}\label{EquivNorms}\begin{split}&\,2\int_{\mathbb{R}^d} D^s u\cdot D^s v=2\int_{\mathbb{R}^d} v(-\Delta)^s u\\=&\, c_{d,s}^2\left[\int_{\mathbb{R}^d}\lim_{\epsilon\to0} \int_{\mathbb{R}^d}v(y)\frac{u(y)-u(x)}{|x-y|^{d+2s}}\chi_\epsilon(x,y)\,dx\,dy+\int_{\mathbb{R}^d}\lim_{\epsilon\to0} \int_{\mathbb{R}^d}v(x)\frac{u(x)-u(y)}{|x-y|^{d+2s}}\chi_\epsilon(x,y)\,dy\,dx\right]\\=&\, c_{d,s}^2\left[\lim_{\epsilon\to0}\int_{\mathbb{R}^d} \int_{\mathbb{R}^d}v(y)\frac{u(y)-u(x)}{|x-y|^{d+2s}}\chi_\epsilon(x,y)\,dx\,dy+\lim_{\epsilon\to0}\int_{\mathbb{R}^d} \int_{\mathbb{R}^d}v(x)\frac{u(x)-u(y)}{|x-y|^{d+2s}}\chi_\epsilon(x,y)\,dx\,dy\right]\\=&\, c_{d,s}^2\int_{\mathbb{R}^d}\int_{\mathbb{R}^d}\frac{(u(x)-u(y))(v(x)-v(y))}{|x-y|^{d+2s}}\,dx\,dy,\end{split}\end{align} where we have used the above definitions together with the Lebesgue and Fubini theorems. 

In this work, we are concerned with the classical fractional Sobolev space $H^s_0(\Omega)$ in a bounded domain $\Omega\subset\mathbb{R}^d$ with Lipschitz boundary, for $0<s<1$, defined as \[H^s_0(\Omega):=\overline{C_c^\infty(\Omega)}^{\norm{\cdot}_{H^s}},\] with \[\norm{u}_{H^s}^2=\norm{u}_{L^2(\mathbb{R}^d)}^2+\norm{D^su}_{L^2(\mathbb{R}^d)}^2,\] where $u$ is extended by 0 in $\mathbb{R}^d\backslash\Omega$, so that this extension is also in $H^s(\mathbb{R}^d)$. By the Sobolev-Poincar\'e inequality (see Theorem 1.7 of \cite{SS1} and Lemma \ref{Poincare}), we may consider the space $H^s_0(\Omega)$ with the following equivalent norms, owing to \eqref{EquivNorms}, \begin{equation}\label{EquivNorms2}\norm{u}_{H^s_0(\Omega)}^2:=\norm{D^su}_{L^2(\mathbb{R}^d)}^2=\frac{c_{d,s}^2}{2}[u]_{s,{\mathbb{R}^d}}^2:=\frac{c_{d,s}^2}{2}\int_{\mathbb{R}^d}\int_{\mathbb{R}^d}\frac{(u(x)-u(y))^2}{|x-y|^{d+2s}}\,dx\,dy.\end{equation}

On the other hand, since the Riesz kernel is an approximation to the identity as $1-s\to0$, the $s$-derivatives approach the classical derivatives as $s\to1$ in appropriate spaces, as it was observed in Rodrigues-Santos \cite{RodriguesSantos}, Comi-Stefani \cite{Comi2} and Belido et al \cite{BellidoCuetoMoraCorral2021CVPDE}. 

We can subsequently denote the dual space of $H^s_0(\Omega)$ by $H^{-s}(\Omega)$ for $0<s\leq1$. Then, by the Sobolev-Poincar\'e inequalities, we have the embeddings \[H^s_0(\Omega)\hookrightarrow L^q(\Omega), \quad\  L^{2^\#}(\Omega)\hookrightarrow H^{-s}(\Omega)=(H^s_0(\Omega))'\] for $1\leq q\leq2^*$, where $2^*=\frac{2d}{d-2s}$ and $2^\#=\frac{2d}{d+2s}$ when $s<\frac{d}{2}$, and if $d=1$, $2^*=q$ for any finite $q$ and $2^\#=q'=\frac{q}{q-1}$ when $s=\frac{1}{2}$ and $2^*=\infty$ and $2^\#=1$ when $s>\frac{1}{2}$. We recall that those embeddings are compact for $1\leq q<2^*$ (see for example, Theorem 4.54 of \cite{Demengel}). In the whole paper we use $2^\#$ to indicate this number that depends on $d\geq1$ and $0<s\leq1$.

We consider the closed convex set \[\mathbb{K}_\psi^s=\{v\in H^s_0(\Omega):v\geq\psi \text{ a.e. in }\Omega\},\] with a given obstacle $\psi$, such that $\mathbb{K}_\psi^s\neq\emptyset$, and the obstacle problem \begin{equation}u\in\mathbb{K}_\psi^s:\quad\langle\mathcal{L}_au,v-u\rangle\geq\langle F,v-u\rangle\quad\forall v\in\mathbb{K}^s_\psi,\end{equation} for $F$ in $H^{-s}(\Omega)$. Here, the nonlocal operator $\mathcal{L}_a:H^s_0(\Omega)\to H^{-s}(\Omega)$ is a generalization of the fractional Laplacian for a measurable, bounded, strictly positive kernel $a:\mathbb{R}^d\times\mathbb{R}^d\to[0,\infty)$ satisfying \eqref{kAcond}, and is defined in the duality sense for $u,v\in H^s_0(\Omega)$, extended by zero outside $\Omega$: \begin{equation}\label{LAdef}\langle\mathcal{L}_au,v\rangle:=\lim_{\varepsilon\to0}\int_{\mathbb{R}^d} \int_{\mathbb{R}^d}\tilde{v}(x)(\tilde{u}(x)-\tilde{u}(y))a(x,y)\chi_\varepsilon(x,y) \,dy\,dx=\,P.V.\int_{\mathbb{R}^d} \int_{\mathbb{R}^d}\tilde{v}(x)(\tilde{u}(x)-\tilde{u}(y))a(x,y) \,dy\,dx.\end{equation} Physically, the operator $\mathcal{L}_a$ corresponds to the class of uniformly irreducible random walks that admit a cycle decomposition with bounded range, bounded length of cycles, and bounded jump rates \cite{KumagaiDeuschelNonsymmetricKernel}.

To better characterize the properties of the operator $\mathcal{L}_a$, in Section 2, we show that the bilinear form \[\mathcal{E}_a(u,v):=\langle\mathcal{L}_au,v\rangle=\,P.V.\int_{\mathbb{R}^d} \int_{\mathbb{R}^d}\tilde{v}(x)(\tilde{u}(x)-\tilde{u}(y))a(x,y)\,dy\,dx\] is a (not necessarily symmetric) Dirichlet form over $H^s_0(\Omega)\times H^s_0(\Omega)$, where $\tilde{u}$ and $\tilde{v}$ are the zero extensions of $u$ and $v$ outside $\Omega$ respectively. This provides us with many known properties of Dirichlet forms that can be applied to the bilinear form $\mathcal{E}_a$, including the truncation property and some regularity results \cite{FukushimaDirichletForms}\cite{MaRockner}. As a corollary, we obtain that $\mathcal{E}_a$ is a closed, coercive, strictly T-monotone and regular Dirichlet form in $H^s_0(\Omega)$. Furthermore, we use the results of the nonlocal vector calculus developed by Du, Gunzburger, Lehoucq and coworkers in \cite{DElia2020unified}, \cite{DuGunzburgerLehoucqZhouNonlocalDiffusion}, \cite{DuGunzburgerLehoucqZhouNonlocalVectorCalculus} and \cite{GunzburgerLehoucq} to show that the fractional operator $\tilde{\mathcal{L}}_A:H^s_0(\Omega)\to H^{-s}(\Omega)$, defined by \begin{equation}\label{LAtildedef}\langle\tilde{\mathcal{L}}_Au,v\rangle:=\int_{\mathbb{R}^d} A(x)D^su\cdot D^sv \,dx,\quad\forall u,v\in H^s_0(\Omega),\end{equation}  corresponds to a nonlocal integral operator $\mathcal{L}_{k_A}$ with a (not necessarily symmetric) singular kernel $k_A(x,y)$ defined by \eqref{kAform}. We were also motivated by the issues raised by Shieh and Spector in \cite{SS2}, in particular their Open Problem 1.10, which by \eqref{EquivNorms} clearly holds when $\tilde{\mathcal{L}}_A=(-\Delta)^s$. Discussing this issue with examples, we give a counterexample in Example \ref{countereg}, showing how interesting is their Open Problem 1.10 for general strictly elliptic and bounded matrices $A$. Then, we conjecture in the \hyperlink{thesentence}{Open Problem} that the kernel $k_A$ corresponding to $\tilde{\mathcal{L}}_A$ has the required property if and only if $\tilde{\mathcal{L}}_A$ is approximately a constant multiple of the fractional Laplacian, up to small bounded perturbations.

In Section 3, we make use of the comparison property to consider obstacle-type problems involving the bilinear form $\mathcal{E}_a$, thereby considering the nonlocal obstacle problem for which we derive results similar to the classical case in $H^1_0(\Omega)$ as in \cite{Stampacchia1965}, \cite{KinderlehrerStampacchia} and \cite{ObstacleProblems}, such as the weak maximum principle and comparison properties. Making use of convergence properties of the fractional derivatives when $s\nearrow1$ to the classical derivatives, as already observed in \cite{RodriguesSantos}, \cite{BellidoCuetoMoraCorral2021CVPDE} and \cite{Comi2}, we show that the solution of the fractional obstacle problem for $\tilde{\mathcal{L}}_A$ converges to the solution of the classical case corresponding to $s=1$.

By considering the approximation of the obstacle problem by semilinear problems using a bounded penalization, in Section 4, we give a direct proof of the Lewy-Stampacchia inequalities for the obstacle-type problems. Here we consider the non-homogeneous data $F=f\in L^{2^\#}(\Omega)$ not only for the one obstacle problem but also for the two obstacles problem and for the N-membranes problem in the nonlocal framework, extending the results of \cite{ServadeiValdinoci2013RMILewyStampacchia}. In particular, we extend the estimates in energy of the difference between the approximating solutions and the solutions of the one and the two obstacle problems, which may be useful for numerical applications such as in \cite{FiniteElement} or \cite{Application}. More important is the use of the Lewy-Stampacchia inequalities that, upon restricting $a$ to the symmetric case, allows the application of the results of \cite{KassmannDyda} to obtain locally the H\"older regularity of the solutions to those three nonlocal obstacle type problems. Such regularity results are weaker than those obtained from the fractional Laplacian (such as in \cite{RosOton2014Regularity}) or other commonly considered nonlocal kernels \cite{Fall2020CVPDE}, since $a$ is in general not a constant multiple of $|x-y|^{-d-2s}$. In this special case, when $\mathcal{L}_a=(-\Delta)^s$, the one obstacle problem can be written for $\bm{f}\in [L^2(\mathbb{R}^d)]^d$ \[u\in\mathbb{K}^s_\psi(\Omega):\int_{\mathbb{R}^d}(D^su-\bm{f})\cdot D^s(v-u)\,dx\geq0\quad\forall v\in\mathbb{K}^s_\psi(\Omega),\] as well as for the corresponding inequalities for the two obstacles and the N-membranes problems. We are then able to use the results of \cite{BiccariWarmaZuazua} together with the Lewy-Stampacchia inequalities to obtain locally regular solutions in the fractional Sobolev space $W^{2s,p}_{loc}(\Omega)$ for $p\geq 2^\#$ and also in $C^1(\Omega)$ for $s>1/2$ and $p>d/(2s-1)$ when $D^s\cdot \bm{f}\in L^p(\mathbb{R}^d)$.

In Section 5, we further consider some properties related to the fractional $s$-capacity extending some classical results of Stampacchia \cite{Stampacchia1965}. We characterize the order dual of $H^s_0(\Omega)$ as the dual space of $L^2_{C_s}(\Omega)$, i.e. the space of quasi-continuous functions with respect to the $s$-capacity which are in absolute value quasi-everywhere dominated by $H^s_0(\Omega)$ functions, extending results of \cite{AttouchPicard}. That dual space corresponds then to the elements of $H^{-s}(\Omega)$ that are also bounded measures, i.e. $(L^2_{C_s}(\Omega))'=H^{-s}(\Omega)\cap M(\Omega)$. Therefore, using the strict T-monotonicity of $\mathcal{L}_a$, we state the Lewy-Stampacchia inequalities in this dual space. This section ends with some new remarks on the relations of the $\mathcal{E}_a$ obstacle problem and the $s$-capacity.


\section{The Anisotropic Non-symmetric Nonlocal Bilinear Form}

\subsection{The nonlocal bilinear form as a coercive Dirichlet form}

Our first main result is to show that the nonlocal bilinear form together with its domain, $(\mathcal{E}_a,H^s_0(\Omega))$, defined as \begin{equation}\label{EAdef2.1}\mathcal{E}_a(u,v):=P.V.\int_{\mathbb{R}^d}\int_{\mathbb{R}^d} \tilde{v}(x)(\tilde{u}(x)-\tilde{u}(y))a(x,y) \,dy\,dx,\end{equation} where $\tilde{u},\tilde{v}$ are the zero extensions of $u,v\in H^s_0(\Omega)$ to $\Omega^c$ and $a:\mathbb{R}^d\times\mathbb{R}^d\to[0,\infty), d\geq1$ is a not necessarily symmetric kernel satisfying \begin{equation}\label{kAcond}a_*c_{d,s}^2\leq a(x,y)|x-y|^{d+2s}\leq a^*c_{d,s}^2\quad\forall x,y\in\mathbb{R}^d,x\neq y,\end{equation} is in fact a regular (not necessarily symmetric) Dirichlet form. This will also imply that the nonlocal bilinear form is also strictly T-monotone and will give us many properties, including Harnack's inequality (see for example \cite{FracLapMaxPrinciple}), H\"older regularity of solutions of equations involving this bilinear form (see for instance \cite{KassmannDeGiorgiNonlocal}, \cite{KassmannHolder}, or \cite{KassmannHarnack}), and other results in stochastic processes (as given in \cite{FukushimaDirichletForms}).

We will begin with a remark on the symmetric case.

\begin{proposition}
\label{UnweightedGreen} For given $u,v\in H^s_0(\Omega)$ and $a:\mathbb{R}^d\times\mathbb{R}^d\to[0,\infty)$ symmetric, we have \begin{multline}\label{UnweightedGreenEq} \int_{\mathbb{R}^d}\int_{\mathbb{R}^d} \tilde{v}(x)(\tilde{u}(x)-\tilde{u}(y))a(x,y)\chi_\varepsilon(x,y) \,dy\,dx\\=\frac{1}{2}\int_{\mathbb{R}^d}\int_{\mathbb{R}^d}(\tilde{u}(x)-\tilde{u}(y))(\tilde{v}(x)-\tilde{v}(y)) a(x,y)\chi_\varepsilon(x,y)\,dy\,dx\end{multline} assuming that the integrands are summable for each fixed $\varepsilon>0$, where we have set $\chi_\varepsilon(x,y)$ as the characteristic function of the set $\{|x-y|>\varepsilon\}$ for $\varepsilon>0$.
\end{proposition}

\begin{proof} We first show the result for $u,v\in C_c^\infty(\Omega)$, and extend it by density to $u,v\in H^s_0(\Omega)$.
The integral term  \[J:=\int_{\mathbb{R}^d} \int_{\mathbb{R}^d}\tilde{v}(x)(\tilde{u}(x)-\tilde{u}(y))a(x,y)\chi_\varepsilon(x,y) \,dy\,dx\] can also be written in the form, using the symmetry of $a$, \[J=\int_{\mathbb{R}^d}\int_{\mathbb{R}^d}\tilde{v}(y) (\tilde{u}(y)-\tilde{u}(x))a(x,y)\chi_\varepsilon(x,y) \,dx\,dy\] 
Then, by Fubini theorem, \[J=-\int_{\mathbb{R}^d}\int_{\mathbb{R}^d} \tilde{v}(y)(\tilde{u}(x)-\tilde{u}(y))a(x,y)\chi_\varepsilon(x,y) \,dy\,dx.\] Taking the sum of this and the first equation above, we obtain the result \[2J=\int_{\mathbb{R}^d}\int_{\mathbb{R}^d}(\tilde{v}(x)-\tilde{v}(y))(\tilde{u}(x)-\tilde{u}(y))a(x,y)\chi_\varepsilon(x,y) \,dy\,dx.\] 
\end{proof}

\begin{remark}\label{ThmHs}
The bilinear form \eqref{EAdef2.1} is also well-defined for $u,v\in H^s(\mathbb{R}^d)$. Indeed, \begin{multline*}
\lim_{\varepsilon\to0}\int_{\mathbb{R}^d}\int_{\mathbb{R}^d}a_*c_{d,s}^2\frac{v(x)(u(x)-u(y))}{|x-y|^{d+2s}}\chi_\varepsilon(x,y)\,dx\,dy\\\leq \lim_{\varepsilon\to0}\int_{\mathbb{R}^d}\int_{\mathbb{R}^d}v(x)(u(x)-u(y))a(x,y)\chi_\varepsilon(x,y)\,dx\,dy\\\leq \lim_{\varepsilon\to0}\int_{\mathbb{R}^d}\int_{\mathbb{R}^d} a^*c_{d,s}^2\frac{v(x)(u(x)-u(y))}{|x-y|^{d+2s}}\chi_\varepsilon(x,y)\,dx\,dy\end{multline*} with, applying Fubini's theorem, \begin{align*}&\lim_{\varepsilon\to0}\int_{\mathbb{R}^d}\int_{\mathbb{R}^d}\frac{v(x)(u(x)-u(y))}{|x-y|^{d+2s}}\chi_\varepsilon(x,y)\,dy\,dx\\=&\lim_{\varepsilon\to0}\int_{\mathbb{R}^d}\int_{\mathbb{R}^d}\frac{v(y)(u(y)-u(x))}{|x-y|^{d+2s}}\chi_\varepsilon(x,y)\,dy\,dx\\=&\frac{1}{2}\int_{\mathbb{R}^d}\int_{\mathbb{R}^d}\frac{(v(x)-v(y))}{|x-y|^{\frac{d}{2}+s}}\frac{(u(x)-u(y))}{|x-y|^{\frac{d}{2}+s}}\,dy\,dx\end{align*} which is integrable by Cauchy-Schwarz inequality when $u,v\in H^s(\mathbb{R}^d)$.
\end{remark}

It is well-known that $H^s_0(\Omega)$ is complete with respect to its norm, as in \eqref{EquivNorms2}, therefore $\mathcal{E}_a$ is \emph{closed}. Furthermore, it can be shown that $(\mathcal{E}_a,H^s_0(\Omega))$ is \emph{regular}, i.e. $H^s_0(\Omega)\cap C_c(\Omega)$ is dense in $H^s_0(\Omega)$ in the $H^s_0(\Omega)$-norm, and dense in $C_c(\Omega)$ with uniform norm. The first density result follows from the density of compactly supported smooth functions in the space $H^s_0(\Omega)$. The second density result follows since $C_c^\infty(\Omega)\subset H^s_0(\Omega)$ and $C_c^\infty(\Omega)$ is dense in $C_c(\Omega)$ with uniform norm, by considering the mollification of any $C_c(\Omega)$ function.

Recall that a coercive closed (not necessarily symmetric) bilinear form $\mathcal{E}$ on $L^2(\Omega)$ is a \emph{Dirichlet form} (\cite{MaRockner} Proposition I.4.7 and equation (4.7) pages 34--35) if and only if the following property holds:
For each $\varepsilon>0$, there exists a real function $\phi_\varepsilon(t),t\in\mathbb{R}$, such that \begin{equation}\label{Markovian1}\phi_\varepsilon(t)=t,\quad\forall t\in[0,1]\,\quad-\varepsilon\leq\phi_\varepsilon(t)\leq 1+\varepsilon,\quad\forall t\in\mathbb{R},\quad0\leq\phi_\varepsilon(t')-\phi_\varepsilon(t)\leq t'-t\text{ whenever }t<t'\end{equation}
\begin{equation}\label{Markovian2}u\in H^s_0(\Omega)\implies\phi_\varepsilon(u)\in H^s_0(\Omega),\quad\begin{cases}\liminf_{\varepsilon\to0}\mathcal{E}(\phi_\varepsilon(u),u-\phi_\varepsilon(u))\geq0,\\\liminf_{\varepsilon\to0}\mathcal{E}(u-\phi_\varepsilon(u),\phi_\varepsilon(u))\geq0.\end{cases}\end{equation}

A classic example of $\phi_\varepsilon$ is the mollification of a cut-off function (see \cite{FukushimaDirichletForms} Example 1.2.1). Specifically, consider a mollifier such as \[j(x)=\begin{cases}\gamma e^{-\frac{1}{1-|x|^2}}&\text{ for }|x|<1\\0&\text{ otherwise},\end{cases}\] where $\gamma$ is a positive constant such that \[\int_{|x|\leq1} j(x)\,dx=1.\] Set $j_\delta(x)=\delta^{-1}j(\delta^{-1}x)$ for $\delta>0$. For any $\varepsilon>0$, consider the function $\psi_\varepsilon(t)=((-\varepsilon)\vee t)\wedge (1+\varepsilon)$ on $\mathbb{R}$ and set $\phi_\varepsilon(t)=j_\delta*\psi_\varepsilon(t)$ for $0<\delta<\varepsilon$. Then our choice of $\phi_\varepsilon$ satisfies \eqref{Markovian1}. Furthermore, it satisfies  $\phi_\varepsilon(t)=1+\varepsilon$ for $t\in[1+2\varepsilon,\infty)$ and $\phi_\varepsilon(t)=-\varepsilon$ for $t\in(-\infty,-2\varepsilon]$, and $|\phi_\varepsilon(t)|\leq |t|$ with $t\phi_\varepsilon(t)\geq0$.

Moreover, $\phi_\varepsilon(t)$ is infinitely differentiable, so for any $u\in C_c^\infty(\Omega)$, $\phi_\varepsilon(u)\in C_c^\infty(\Omega)$. Since $C_c^\infty(\Omega)$ is dense in $H^s_0(\Omega)$, we can extend by density any results obtained, thereby obtaining $\phi_\varepsilon(u)\in H^s_0(\Omega)$ for all $u\in H^s_0(\Omega)$.

With this $\phi_\varepsilon$, recalling that $a\geq0$ (as in section II.2(d) of \cite{MaRockner}), by Fatou's lemma,  \begin{align*}&\liminf_{\varepsilon\to0}\mathcal{E}_a(\phi_\varepsilon(u),u-\phi_\varepsilon(u))\\\geq&\,P.V.\int_{\mathbb{R}^d}\int_{\mathbb{R}^d}\liminf_{\varepsilon\to0}[\widetilde{u(x)-\phi_\varepsilon(u(x))}][\widetilde{\phi_\varepsilon(u(x))}-\widetilde{\phi_\varepsilon(u(y))}]a(x,y)\,dy\,dx\\\geq&\,P.V.\int_{\mathbb{R}^d}\int_{\mathbb{R}^d}\liminf_{\varepsilon\to0}[\widetilde{u(x)-\phi_\varepsilon(u(x))}][-\widetilde{\phi_\varepsilon(u(y))}]a(x,y)\,dy\,dx\qquad\qquad\qquad\text{ (since }\phi_\varepsilon(u)[u-\phi_\varepsilon(u)]\geq0)
\\\geq&\,P.V.\int_{\mathbb{R}^d}\int_{\mathbb{R}^d}\liminf_{\varepsilon\to0}\Big\{-[\widetilde{u(x)-\phi_\varepsilon(u(x))}](-\varepsilon)\chi_{\{u(x)<0\}} +0\chi_{\{0\leq u(x)\leq1\}}\\&\qquad\qquad+(-\varepsilon)(-1-\varepsilon)\chi_{\{u(x)>1\}}\Big\}a(x,y)\,dy\,dx\\=&\,0\end{align*} since $u-\phi_\varepsilon(u)\leq0$ for $u\leq0$, $\phi_\varepsilon(t)=t$ for $t\in[0,1]$, and $u-\phi_\varepsilon(u)\geq u-1-\varepsilon\geq -\varepsilon$ for $u\geq1$ respectively. 
Similarly, taking $\liminf_{\varepsilon\to0}$ in
\[\mathcal{E}_a(u-\phi_\varepsilon(u),\phi_\varepsilon(u))=P.V.\int_{\mathbb{R}^d}\int_{\mathbb{R}^d}\widetilde{\phi_\varepsilon(u(x))}\left\{\widetilde{[u(x)-\phi_\varepsilon(u(x))]}-\widetilde{[u(y)-\phi_\varepsilon(u(y))]}\right\}a(x,y)\,dx\,dy,\]
we conclude that $\mathcal{E}_a$ is a Dirichlet form. 
Therefore, we obtained the following theorem:

\begin{theorem}\label{DirichletFormThm}
$(\mathcal{E}_a,H^s_0(\Omega))$ is a closed, regular Dirichlet form, which is also bounded and coercive.
\end{theorem}

\begin{proof}
It remains only to show that the bilinear form $\mathcal{E}_a:H^s_0(\Omega)\times H^s_0(\Omega)\to\mathbb{R}$ is bounded \begin{align*}\mathcal{E}_a(u,v)&\leq a^*c_{d,s}^2\,P.V.\int_{\mathbb{R}^d} \int_{\mathbb{R}^d}\tilde{v}(x)(\tilde{u}(x)-\tilde{u}(y))|x-y|^{-d-2s} \,dy\,dx\\&=\frac{1}{2}a^*c_{d,s}^2\int_{\mathbb{R}^d}\int_{\mathbb{R}^d}(\tilde{u}(x)-\tilde{u}(y))(\tilde{v}(x)-\tilde{v}(y))|x-y|^{-d-2s} \,dy\,dx\\&\leq a^*c_{d,s}^2\left(\frac{1}{2}\int_{\mathbb{R}^d}\int_{\mathbb{R}^d}\frac{(\tilde{u}(x)-\tilde{u}(y))^2}{|x-y|^{d+2s}}\,dx\,dy\right)^{1/2}\left(\frac{1}{2}\int_{\mathbb{R}^d}\int_{\mathbb{R}^d}\frac{(\tilde{v}(x)-\tilde{v}(y))^2}{|x-y|^{d+2s}}\,dx\,dy\right)^{1/2}\\&=a^*\norm{D^su}_{L^2(\Omega)}\norm{D^sv}_{L^2(\Omega)}\\&\leq a^*\norm{u}_{H^s_0(\Omega)}\norm{v}_{H^s_0(\Omega)}\quad\forall u,v\in H^s_0(\Omega)\end{align*} by the Cauchy-Schwarz inequality, and coercive \[\mathcal{E}_a(u,u)\geq\frac{1}{2}a_*c_{d,s}^2\int_{\mathbb{R}^d}\int_{\mathbb{R}^d}(\tilde{u}(x)-\tilde{u}(y))^2|x-y|^{-d-2s} \,dy\,dx= a_*\norm{D^su}^2_{L^2(\mathbb{R}^d)}=a_*\norm{u}^2_{H^s_0(\Omega)}.\] 
\end{proof}

From this theorem, we obtain that $\mathcal{E}_a$ possesses the property of unit contraction. Indeed, by Proposition 4.3 and Theorem 4.4 of \cite{MaRockner}, we have the following corollary.

\begin{corollary}
For the regular Dirichlet form $\mathcal{E}_a$ the following properties hold:
\begin{enumerate}[label=(\alph*)]
\item the unit contraction acts on $\mathcal{E}_a$, i.e. $v:=(0\vee u)\wedge1$ satisfies $\mathcal{E}_a(v,v)\leq \mathcal{E}_a(u,u)$;\item the normal contraction acts on $\mathcal{E}_a$, i.e. suppose $v$ satisfies \[|v(x)-v(y)|\leq|u(x)-u(y)|,\quad x,y\in\mathbb{R}^d\]\[|v(x)|\leq|u(x)|,\quad x\in\mathbb{R}^d,\] then $v\in H^s_0(\Omega)$ and $\mathcal{E}_a(v,v)\leq \mathcal{E}_a(u,u)$.
\end{enumerate}
\end{corollary}
This result in fact follows from the fact that $\mathcal{E}_a$ is a Dirichlet form, and holds even if it is not regular.

Since $\mathcal{E}_a$ is a regular Dirichlet form, it is possible to consider truncations, so we can introduce the positive and negative parts of $v$
\[v^+\equiv v\vee0\quad\text{ and }\quad v^-\equiv-v\vee0=-(v\wedge0),\] and we have the Jordan decomposition of $v$ given by \[v=v^+-v^-\quad\text{ and }\quad |v|\equiv v\vee(-v)=v^++v^-\] and the useful identities \[u\vee v=u+(v-u)^+=v+(u-v)^+,\]\[u\wedge v=u-(u-v)^+=v-(v-u)^+.\] It is well-known that such operations are closed in $H^s_0(\Omega)$ for $0<s\leq1$.

\subsection{The strict T-monotonicity of $\mathcal{E}_a$}

Since the assumptions on the kernel $a$ imply, in particular, that it is non-negative, we can easily prove the following important property.

\begin{theorem}
\label{BTmonotone}
$\mathcal{E}_a$ is also strictly T-monotone in the following sense:
$\mathcal{L}_a:H^s_0(\Omega)\to H^{-s}(\Omega)$ defined by \begin{equation}\label{LaEadef}\langle\mathcal{L}_au,v\rangle=\mathcal{E}_a(u,v),\end{equation} satisfies \[\langle\mathcal{L}_av,v^+\rangle>0\quad\forall v\in H^s_0(\Omega)\text{ such that } v^+\neq0.\]
\end{theorem}

\begin{proof}

$\mathcal{L}_a$ is strictly T-monotone because \begin{align*}\mathcal{E}_a(v^-,v^+)=&\,P.V.\int_{\mathbb{R}^d}\int_{\mathbb{R}^d} \tilde{v}^+(x)(\tilde{v}^-(x)-\tilde{v}^-(y))a(x,y)\,dx\,dy\\=&-P.V.\int_{\mathbb{R}^d}\int_{\mathbb{R}^d} \tilde{v}^+(x)\tilde{v}^-(y)a(x,y)\,dx\,dy\\\leq&\,0,\end{align*} since $v^+(x)v^-(x)=0$ as $v^+$ and $v^-$ cannot both be nonzero at the same point $x$, and $v^+,v^-,a\geq0$. Therefore, since $a(x,y)|x-y|^{d+2s}\geq a_*c_{d,s}^2$ with $a_*>0$, \begin{align*}\langle\mathcal{L}_Av,v^+\rangle&=\mathcal{E}_a(v,v^+)\\&=\mathcal{E}_a\left(v^+,v^+\right)-\mathcal{E}_a\left(v^-,v^+\right)\\&\geq a_*\int_{\mathbb{R}^d}|D^sv^+|^2\end{align*} which is strictly greater than 0 if $v^+\neq 0$.
\end{proof}

\subsection{The fractional bilinear form as a nonlocal bilinear form}

In this section, we consider the linear operator $\tilde{\mathcal{L}}_A$ defined with the fractional derivatives $D^s$ by the continuous fractional bilinear form \begin{equation}\tag{\ref{LAtildedef}}\langle\tilde{\mathcal{L}}_Au,v\rangle:=\int_{\mathbb{R}^d} A(x)D^su\cdot D^sv \,dx,\quad\forall u,v\in H^s_0(\Omega)\end{equation} for a matrix $A$ with bounded and measurable coefficients. The integral in \eqref{LAtildedef} is well defined and we are going to show that this bilinear form can be rewritten with a measurable kernel $k_A:\mathbb{R}^d\times\mathbb{R}^d, d\geq1$, as \begin{equation}\label{EAtildedef2.1}\mathcal{E}_{k_A}(u,v):=P.V.\int_{\mathbb{R}^d}\int_{\mathbb{R}^d} \tilde{v}(x)(\tilde{u}(x)-\tilde{u}(y))k_A(x,y) \,dy\,dx,\end{equation} where $\tilde{u},\tilde{v}$ are the zero extensions of $u,v\in C_c^\infty(\Omega)$ to $\Omega^c$. 

\begin{theorem}
\label{OpenPb10Hs}
Given a matrix $A:\mathbb{R}^d\to\mathbb{R}^{d\times d}$ with bounded and measurable coefficients, there exists a kernel $k_A(x,y)$ independent of $u,v$ satisfying \begin{equation}\int_{\mathbb{R}^d} A(x)D^s u(x)\cdot D^sv(x)\,dx=P.V.\int_{\mathbb{R}^d}\int_{\mathbb{R}^d} v(x)(u(x)-u(y))k_A(x,y) \,dy\,dx\label{OpenPb10HsEq}\end{equation} for all $u,v\in C_c^\infty(\mathbb{R}^d)$, where $k_A(x,y)$ is given by \begin{equation}\label{kAform}k_A(x,y)=c_{d,s}^2\,P.V.\int_{\mathbb{R}^d} A(z) \frac{y-z}{|y-z|^{d+s+1}}\cdot \frac{z-x}{|z-x|^{d+s+1}}\,dz\quad\text{ for }x\neq y.\end{equation}
\end{theorem}

\begin{proof}

Expanding the fractional bilinear form, we have, setting for simplicity $\int_{\mathbb{R}^d}$ as $\int$, \begin{align*}
&\,\int  A(z)D^su(z)\cdot D^sv(z)\,dz\\=&\, c_{d,s}^2\int A(z)\left[\int(u(y)-u(z)) \frac{y-z}{|y-z|^{d+s+1}}\,dy\right]\cdot\left[\int (v(x)-v(z))\frac{(x-z)}{|z-x|^{d+s+1}}\,dx\right]\,dz\\=&\, c_{d,s}^2\iiint\lim_{\varepsilon,\delta,\eta\to0}\left[A(z)(u(y)-u(z)) \frac{(y-z)\chi_\delta(y,z)}{|y-z|^{d+s+1}}\cdot (v(x)-v(z))\frac{(x-z)\chi_\varepsilon(x,z)}{|z-x|^{d+s+1}}\chi_\eta(x,y)\right]\,dx\,dy\,dz\\=&\, c_{d,s}^2\lim_{\varepsilon,\delta,\eta\to0}\iiint A(z)(u(y)-u(z)) \frac{(y-z)\chi_\delta(y,z)}{|y-z|^{d+s+1}}\cdot (v(x)-v(z))\frac{(x-z)\chi_\varepsilon(x,z)}{|z-x|^{d+s+1}}\chi_\eta(x,y)\,dx\,dy\,dz,\end{align*}where $\chi_\eta(x,y)$ is the characteristic function on the set $\{|x-y|>\eta\}$ and similarly defined for $\chi_\varepsilon$ and $\chi_\delta$.
The limit can be exchanged with the integral by the Fubini and Lebesgue theorems because the integrand is Lebesgue integrable. Therefore, \begin{align}
&\,\int  A(z)D^su(z)\cdot D^sv(z)\,dz\nonumber\\=&\, c_{d,s}^2\lim_{\varepsilon,\delta,\eta\to0}\iiint A(z)(u(y)-u(z)) \frac{(y-z)\chi_\delta(y,z)}{|y-z|^{d+s+1}}\cdot v(x)\frac{(x-z)\chi_\varepsilon(x,z)}{|z-x|^{d+s+1}}\chi_\eta(x,y)\,dx\,dy\,dz\nonumber\\&\, - c_{d,s}^2\lim_{\varepsilon,\delta,\eta\to0}\iiint A(z)(u(y)-u(z)) \frac{(y-z)\chi_\delta(y,z)}{|y-z|^{d+s+1}}\cdot v(z)\frac{(x-z)\chi_\varepsilon(x,z)}{|z-x|^{d+s+1}}\chi_\eta(x,y)\,dx\,dy\,dz\nonumber\\=&\, c_{d,s}^2\lim_{\varepsilon,\delta,\eta\to0}\iiint A(z)(u(y)-u(z)) \frac{(y-z)\chi_\delta(y,z)}{|y-z|^{d+s+1}}\cdot v(x)\frac{(x-z)\chi_\varepsilon(x,z)}{|z-x|^{d+s+1}}\chi_\eta(x,y)\,dx\,dy\,dz\nonumber\\&\, - c_{d,s}^2\lim_{\varepsilon,\delta\to0}\iint A(z)(u(y)-u(z)) \frac{(y-z)\chi_\delta(y,z)}{|y-z|^{d+s+1}}\cdot v(z)\left[\int\lim_{\eta\to0}\frac{(x-z)\chi_\varepsilon(x,z)}{|z-x|^{d+s+1}}\chi_\eta(x,y)\,dx\right]\,dy\,dz\nonumber\\=&\, c_{d,s}^2\lim_{\varepsilon,\delta,\eta\to0}\iiint A(z)(u(y)-u(z)) \frac{(y-z)\chi_\delta(y,z)}{|y-z|^{d+s+1}}\chi_\eta(x,y)\cdot v(x)\frac{(x-z)\chi_\varepsilon(x,z)}{|z-x|^{d+s+1}}\,dx\,dy\,dz,\label{expandedform}\end{align} using the fact that $\int\frac{(x-z)\chi_\varepsilon(x,z)}{|x-z|^{d+s+1}}\,dx=0$ for all $\varepsilon>0$. Note that we can exchange the limit in $\eta$ with the integrals because the functions are Lebesgue integrable as $|x-y|\to0$.

Adding and subtracting $u(x)$, as $u,v\in C_c^\infty(\mathbb{R}^d)$, we have 
\begin{align}
&\,\int  A(z)D^su(z)\cdot D^sv(z)\,dz\nonumber\\=&\, c_{d,s}^2\lim_{\varepsilon,\delta,\eta\to0}\iiint A(z)[(u(y)-u(x))+(u(x)-u(z))] \frac{(y-z)\chi_\delta(y,z)}{|y-z|^{d+s+1}}\cdot v(x)\frac{(x-z)\chi_\varepsilon(x,z)}{|z-x|^{d+s+1}}\chi_\eta(x,y)\,dx\,dy\,dz\nonumber\\=&\, c_{d,s}^2\lim_{\varepsilon,\delta,\eta\to0}\iiint A(z)(u(y)-u(x)) \frac{(y-z)\chi_\delta(y,z)}{|y-z|^{d+s+1}}\cdot v(x)\frac{(x-z)\chi_\varepsilon(x,z)}{|z-x|^{d+s+1}}\chi_\eta(x,y)\,dx\,dy\,dz\nonumber\\&\,+ c_{d,s}^2\lim_{\varepsilon,\delta,\eta\to0}\int A(z) \left[\int\frac{(y-z)\chi_\delta(y,z)}{|y-z|^{d+s+1}}\cdot\left[\int(u(x)-u(z))v(x)\frac{(x-z)\chi_\varepsilon(x,z)}{|z-x|^{d+s+1}}\chi_\eta(x,y)\,dx\right]\,dy\right]\,dz\nonumber\\=&\, c_{d,s}^2\lim_{\varepsilon,\delta,\eta\to0}\iiint A(z)(u(y)-u(x)) \frac{(y-z)\chi_\delta(y,z)}{|y-z|^{d+s+1}}\cdot v(x)\frac{(x-z)\chi_\varepsilon(x,z)}{|z-x|^{d+s+1}}\chi_\eta(x,y)\,dx\,dy\,dz\nonumber\\&\,+ c_{d,s}^2\lim_{\varepsilon,\delta\to0}\int A(z) \left[\int\frac{(y-z)\chi_\delta(y,z)}{|y-z|^{d+s+1}}\cdot\left[\int\lim_{\eta\to0}(u(x)-u(z))v(x)\frac{(x-z)\chi_\varepsilon(x,z)}{|z-x|^{d+s+1}}\chi_\eta(x,y)\,dx\right]\,dy\right]\,dz\nonumber\\=&\, c_{d,s}^2\lim_{\varepsilon,\delta,\eta\to0}\iiint A(z)(u(y)-u(x)) \frac{(y-z)\chi_\delta(y,z)}{|y-z|^{d+s+1}}\cdot v(x)\frac{(x-z)\chi_\varepsilon(x,z)}{|z-x|^{d+s+1}}\chi_\eta(x,y)\,dx\,dy\,dz,\label{tripleintegral}\end{align} 
where we make use of the fact that $\int\frac{(y-z)\chi_\delta(y,z)}{|y-z|^{d+s+1}}\cdot f(z)\,dy=0$ for all $\delta>0$ for any finite function $f(z)$. Once again, the limit in $\eta$ can be interchanged with triple integrals, because the factor $\frac{(y-z)\chi_\delta(y,z)}{|y-z|^{d+s+1}}$ is integrable for $\delta>0$. Also, the function \[f(z):=\lim_{\eta\to0}\int(u(x)-u(z))\chi_\eta(x,y)v(x)\frac{(x-z)\chi_\varepsilon(x,z)}{|z-x|^{d+s+1}}\,dx\] is a finite function of $z$, because the integrand has a singularity only at $x=z$ and we have introduced the characteristic function $\chi_\varepsilon(x,z)$. Furthermore, the Lipschitz continuity of $u$ guarantees that the singularity is removable, since we have the factor $u(x)-u(z)$. This is also the reason why we used the first expression for the expansion of $D^su$, rather than the second one, which will only give us a factor of $u(x)$ when we add and subtract $u(x)$. Therefore, we can take the limit $\eta\to0$, so that it is just a function of $z$.

Next, we apply Fubini's theorem, since the integrand is Lebesgue integrable for fixed $\varepsilon,\delta,\eta>0$. Therefore, 
\begin{multline*}\int  A(z)D^su(z)\cdot D^sv(z)\,dz\\=c_{d,s}^2\lim_{\varepsilon,\delta,\eta\to0}\iiint A(z)(u(y)-u(x)) \frac{(y-z)\chi_\delta(y,z)}{|y-z|^{d+s+1}}\cdot v(x)\frac{(x-z)\chi_\varepsilon(x,z)}{|z-x|^{d+s+1}}\chi_\eta(x,y)\,dz\,dy\,dx.\end{multline*} 

Finally, regarding this limit as a double limit, in $\eta$ and separately in $\varepsilon$ and $\delta$, which exists, we can consider the iterated limit in the following form \begin{multline*}\int  A(z)D^su(z)\cdot D^sv(z)\,dz\\=\lim_{\eta\to0}\iint(u(x)-u(y))v(x)\left[c_{d,s}^2\lim_{\varepsilon,\delta\to0}\int A(z) \frac{(y-z)\chi_\delta(y,z)}{|y-z|^{d+s+1}}\cdot \frac{(z-x)\chi_\varepsilon(x,z)}{|z-x|^{d+s+1}}\,dz\right]\chi_\eta(x,y)\,dy\,dx\end{multline*} where we may interpret the term in the parentheses as the Cauchy principal value about the singularities $z=x$ and $z=y$, i.e. as a function in $x,y$ defined for $x\neq y$, by \begin{equation}\tag{\ref{kAform}}k_A(x,y)=c_{d,s}^2\,P.V.\int A(z) \frac{y-z}{|y-z|^{d+s+1}}\cdot \frac{z-x}{|z-x|^{d+s+1}}\,dz.\end{equation}
\end{proof}


\begin{remark}Note that in general, $k_A$ is neither translation nor rotation invariant, unlike the case for the fractional Laplacian. In particular, $k_A$ may not have the form $j(x-y)|x-y|^{-d-2s}$. Therefore, the kernel $k_A$ may have relevance for non-homogeneous, non-isotropic and nonlocal problems. Even in the case for $A$ being the constant coefficient matrix when $k_A$ is translation invariant, it may not be rotation invariant, unless if $A$ is a constant multiple of the identity matrix. 
\end{remark}

\begin{remark}\label{kAcoer}
Suppose that the matrix $A$ is given by $\alpha\mathbbm{I}$ for a strictly positive finite constant $\alpha$ and the identity matrix $\mathbbm{I}$. Then, by \eqref{EquivNorms} and Proposition \ref{UnweightedGreen}, $\tilde{\mathcal{L}}_{\alpha\mathbbm{I}}$ defined by \eqref{LAtildedef} can also be defined by a symmetric bilinear form as in \eqref{OpenPb10HsEq} with a kernel $\alpha$ given by \[\alpha(x,y)=\frac{\alpha c_{d,s}^2}{|x-y|^{d+2s}}\] which is, up to a constant, the kernel of the fractional Laplacian and satisfies \eqref{kAcond}. 

However, we observe that this representation may not be unique, and $k_{\alpha\mathbbm{I}}$ may not be equal to $\alpha c_{d,s}^2|x-y|^{-d-2s}$. Indeed, consider an unbounded nonzero $L^2$-integrable function $h(x):\mathbb{R}^d\to\mathbb{R}$ which integrates to 0 over $\mathbb{R}^d$ and has support outside $\Omega$. Let $a(x,y)$ be a kernel satisfying \eqref{kAcond} and define \[\tilde{a}(x,y)=a(x,y)+h(x)h(y),\] which is possible since the kernel is defined over all $(\mathbb{R}^d\times\mathbb{R}^d)\backslash\{x=y\}$ and the integrability of $h$ means that $\mathcal{L}_{\tilde{a}}$ is well defined. Since $\int h=0$ by the construction of $h$ and, for any $u,v\in C_c^\infty(\Omega)$, $\int \tilde{u}h=0$ since they have disjoint supports, we have \begin{align*}\langle\mathcal{L}_{\tilde{a}}u,v\rangle=&\,P.V.\int_{\mathbb{R}^d}\int_{\mathbb{R}^d}\tilde{v}(x)(\tilde{u}(x)-\tilde{u}(y))(a(x,y)+h(x)h(y))\,dy\,dx\\=&\,P.V.\int_{\mathbb{R}^d}\int_{\mathbb{R}^d}\tilde{v}(x)(\tilde{u}(x)-\tilde{u}(y))a(x,y)\,dy\,dx\\&+P.V.\int_{\mathbb{R}^d}\int_{\mathbb{R}^d}\tilde{u}(x)\tilde{v}(x)h(x)h(y)\,dy\,dx-P.V.\int_{\mathbb{R}^d}\int_{\mathbb{R}^d}\tilde{u}(y)\tilde{v}(x)h(x)h(y)\,dy\,dx\\=&\,P.V.\int_{\mathbb{R}^d}\int_{\mathbb{R}^d}\tilde{v}(x)(\tilde{u}(x)-\tilde{u}(y))a(x,y)\,dy\,dx\\&+P.V.\int_{\mathbb{R}^d}\tilde{u}(x)\tilde{v}(x)h(x)\left[\int_{\mathbb{R}^d}h(y)\,dy\right]\,dx-P.V.\int_{\mathbb{R}^d}\tilde{v}(x)h(x)\left[\int_{\mathbb{R}^d}\tilde{u}(y)h(y)\,dy\right]\,dx\\=&\,P.V.\int_{\mathbb{R}^d}\int_{\mathbb{R}^d}\tilde{v}(x)(\tilde{u}(x)-\tilde{u}(y))a(x,y)\,dy\,dx=\langle\mathcal{L}_au,v\rangle\end{align*}

This example gives a class of non-uniqueness for the representation of the kernel. There may be more similar classes, and it will be interesting to know a characterization for the equivalent class of kernels. Furthermore, even if the kernel $a$ satisfies the compatibility condition \eqref{kAcond}, since $h$ may change sign and may be unbounded, our construction of the kernel $\tilde{a}$ may not satisfy the compatibility condition \eqref{kAcond}, nor the weaker compatibility conditions \begin{equation}\label{kernelcondgeneral}\begin{cases}a_*\leq a(x,y)|x-y|^{d+2s}\leq a^*,&\quad|x-y|\leq1\\a(x,y)\leq M|x-y|^{-d-s'},&\quad|x-y|>1\end{cases}\end{equation} for some $a_*,a^*,M,s'>0$, as given in equation (1.11) of \cite{SS2}.
\end{remark}

However, \eqref{kAcond} is not satisfied for the kernel $k_A$ for a general matrix $A$. Indeed, we can construct a numerical counterexample as follows.
\begin{example}\label{countereg}In $d=1$, suppose $s=0.8$, and consider the matrix $A=\alpha(x)$ where $\alpha(x)=0.01+50H(x)$ for the smooth approximation $H$ of the characteristic function of the interval $[1,1.5]$, such that $H(x)=1$ in $[1,1.5]$ and less than $0.0001$ outside $[0.9,1.6]$. Then \[k_{A}(-0.5,0.5)=0.01c_{1,0.8}^2P.V.\int_{\mathbb{R}}\frac{0.5-z}{|0.5-z|^{2.8}} \frac{z+0.5}{|z+0.5|^{2.8}}\,dz+50c_{1,0.8}^2P.V.\int_{\mathbb{R}}H(z) \frac{0.5-z}{|0.5-z|^{2.8}} \frac{z+0.5}{|z+0.5|^{2.8}}\,dz.\] Observe that the function \[\kappa_{1,0.8}(-0.5,0.5,z):=\frac{0.5-z}{|0.5-z|^{2.8}} \frac{z+0.5}{|z+0.5|^{2.8}}\] has the shape \begin{center}\includegraphics[width=0.4\textwidth]{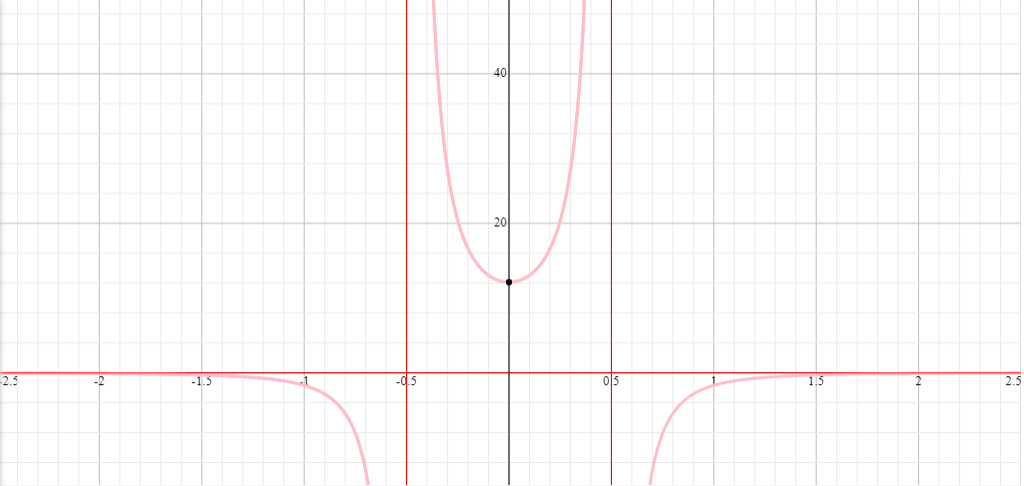}\\{\tiny Created in Symbolab}\end{center} is integrable but not absolutely integrable, and is strictly increasing and strictly negative in the interval $[0.9,1.6]$ with values (computed in Wolfram Alpha) \[\kappa_{1,0.8}(-0.5,0.5,0.9)=-2.839,\quad\kappa_{1,0.8}(-0.5,0.5,1.5)=-0.287,\quad\int_{\mathbb{R}}\kappa_{1,0.8}(-0.5,0.5,z)\,dz\approx30<100.\] Then, computing \eqref{kAform} for  \begin{align*}k_{A}(-0.5,0.5)&=0.01c_{1,0.8}^2P.V.\int_{\mathbb{R}}\frac{0.5-z}{|0.5-z|^{2.8}} \frac{z+0.5}{|z+0.5|^{2.8}}\,dz+50c_{1,0.8}^2P.V.\int_{\mathbb{R}}H(z) \frac{0.5-z}{|0.5-z|^{2.8}} \frac{z+0.5}{|z+0.5|^{2.8}}\,dz\\&<0.01c_{1,0.8}^2(100)+50c_{1,0.8}^2P.V.\int_1^{1.5} \frac{0.5-z}{|0.5-z|^{2.8}} \frac{z+0.5}{|z+0.5|^{2.8}}\,dz\\&\quad+50(0.0001)c_{1,0.8}^2P.V.\int_{\mathbb{R}\backslash[0.9,1.6]} \frac{0.5-z}{|0.5-z|^{2.8}} \frac{z+0.5}{|z+0.5|^{2.8}}\,dz\\&=c_{1,0.8}^2+50c_{1,0.8}^2P.V.\int_1^{1.5} \frac{0.5-z}{|0.5-z|^{2.8}} \frac{z+0.5}{|z+0.5|^{2.8}}\,dz\\&\quad+0.005c_{1,0.8}^2\left(P.V.\int_{\mathbb{R}} \frac{0.5-z}{|0.5-z|^{2.8}} \frac{z+0.5}{|z+0.5|^{2.8}}\,dz-P.V.\int_{0.9}^{1.6} \frac{0.5-z}{|0.5-z|^{2.8}} \frac{z+0.5}{|z+0.5|^{2.8}}\,dz\right)\\&<c_{1,0.8}^2+50c_{1,0.8}^2P.V.\int_1^{1.5} \kappa_{1,0.8}(-0.5,0.5,1.5)\,dz\\&\quad+0.005c_{1,0.8}^2\left(P.V.\int_{\mathbb{R}} \kappa_{1,0.8}(-0.5,0.5,z)\,dz-P.V.\int_{0.9}^{1.6} \kappa_{1,0.8}(-0.5,0.5,0.9)\,dz\right)\\&< c_{1,0.8}^2+50c_{1,0.8}^2(-0.28)(0.5)+0.005c_{1,0.8}^2(100-(-2.84)(1.6-0.9))=-5.49c_{1,0.8}^2<0\end{align*} which contradicts \eqref{kAcond}. Compare with Open Problem 1.10 of \cite{SS2}.\end{example}\qed

The Theorem \ref{OpenPb10Hs} and the last two remarks were inspired by the Open Problem 1.10 of \cite{SS2}, which asked if, given a symmetric matrix $A$ satisfying \eqref{Acoer}, it is possible to find a kernel $k_A$ satisfying \eqref{kernelcondgeneral} such that \eqref{OpenPb10HsEq} holds. Complementing this open problem, we propose the following conjecture (see also Remark 3.1 of \cite{delia2021analysis}): \begin{openpb}\label{OpenProblem}\hypertarget{thesentence}{ } Suppose $A:\mathbb{R}^d\to\mathbb{R}^{d\times d}$ is a bounded, measurable and strictly elliptic matrix such that \begin{equation}\label{Acoer}a_* |z|^2\leq A(x)z\cdot z \text{ and } A(x)z\cdot z^*\leq a^* |z||z^*|\end{equation} for some $a_*,a^*>0$ for all $x\in\mathbb{R}^d$ and all $z,z^*\in\mathbb{R}^d$. Let $k_A$ be a corresponding kernel which is continuous outside the diagonal $x=y$ and satisfies \[\int_{\mathbb{R}^d} A(x)D^s u(x)\cdot D^sv(x)\,dx=P.V.\int_{\mathbb{R}^d}\int_{\mathbb{R}^d} v(x)(u(x)-u(y))k_A(x,y) \,dy\,dx\] for all $u,v\in C_c^\infty(\mathbb{R}^d)$. Then, there exists an equivalent kernel $k_A$ satisfying \[a_*\leq k_A(x,y)|x-y|^{d+2s}\leq a^*\quad\forall x,y\in\mathbb{R}^d,x\neq y\] for some $a_*,a^*>0$ if and only if $A$ is a bounded small perturbation $\tilde{\alpha}(x)$ of the identity matrix (up to a positive constant $\alpha$), i.e. $A=(\alpha+\tilde{\alpha}(x))\mathbbm{I}$ for some strictly positive finite constant $\alpha>>\sup_x|\tilde{\alpha}(x)|$. \end{openpb}

\section{The Nonlocal and the Fractional Obstacle Problem}

We start by considering the linear form for $F\in H^{-s}(\Omega)$ defined by \[\langle F,v\rangle = \int_\Omega f_\#v+\int_{\mathbb{R}^d} \bm{f}\cdot D^sv\] for $v\in H^s_0(\Omega)$, $\bm{f}=(f_1,\dots,f_d)\in [L^2(\mathbb{R}^d)]^d$ and $f_\#\in L^{2^\#}(\Omega)$ where $2^\#=\frac{2d}{d+2s}$ when $s<\frac{d}{2}$, and if $d=1$, $2^\#=q$ for any finite $q>1$ when $s=\frac{1}{2}$ and $2^\#=1$ when $s>\frac{1}{2}$. By the Riesz representation theorem, we have \[\exists! \phi\in H^s_0(\Omega): \int_{\mathbb{R}^d} D^s \phi\cdot D^sv=\langle -D^s\cdot D^s\phi,v\rangle=\langle F,v\rangle,\quad\forall v\in H^s_0(\Omega).\] Therefore, $F\in H^{-s}(\Omega)$ may be given by $F=-D^s\cdot D^s\phi=-D^s\cdot \bm{g}$ for some $\bm{g}=(g_1,\dots, g_d)\in[L^2(\mathbb{R}^d)]^d$ and, by the Sobolev-Poincar\'e inequality, it satisfies \[\norm{\bm{g}}_{[L^2(\mathbb{R}^d)]^d}=\norm{F}_{H^{-s}(\Omega)}\leq C_S\norm{f_\#}_{L^{2^\#}(\Omega)}+\norm{\bm{f}}_{[L^2(\mathbb{R}^d)]^d}.\]
In order for \[F\sim f_\#-D^s\cdot\bm{f}\] to lie in the positive cone of $H^{-s}(\Omega)$, it is enough for $f_\#$ to be non-negative almost everywhere in $\Omega$ and $D^s\cdot\bm{f}\leq0$ in the distributional sense in $\mathbb{R}^d$.

We next recall the following useful inequalities.


\begin{lemma}[Sobolev-Poincar\'e inequality, Theorem 1.8 of \cite{SS1}]\label{Sobolev} Let $s\in(\sigma,1)$ such that $s<\frac{d}{2}$ and $\sigma>0$. Then there exists a constant $C_S=C(d,s)>0$ such that \[\norm{u}_{L^{2^*}(\Omega)}\leq C_S\norm{D^su}_{L^2(\mathbb{R}^d)}\] for all $u\in H^s_0(\Omega)$, where $2^*=\frac{2d}{d-2s}>0$. 
\end{lemma}



\begin{lemma}[Fractional Poincar\'e inequality, Theorem 2.9 of \cite{BellidoCuetoMoraCorral2021CVPDE}]\label{Poincare}
Let $s\in(0,1)$. Then there exists a constant $C_P=C(d,\Omega)/s>0$ such that \[\norm{u}_{L^2(\Omega)}\leq C_P\norm{D^su}_{L^2(\mathbb{R}^d)}\] for all $u\in H^s_0(\Omega)$. 
\end{lemma}

\subsection{The nonlocal obstacle problem and its properties}

As a consequence of the properties of the bilinear form $\mathcal{E}_a$ defined by \eqref{EAdef2.1} in $H^s_0(\Omega)$ for $0<s<1$, we can derive classical properties of the fractional obstacle problems, following most of the approach of Section 4:5 in \cite{ObstacleProblems}.

\begin{theorem}[Obstacle Problem]
\label{ObsPbThm}Let $\Omega\subset\mathbb{R}^d$ be a open bounded domain with Lipschitz boundary, $f_\#\in L^{2^\#}(\Omega)$, $\bm{f}\in [L^2(\mathbb{R}^d)]^d$, and $a:\mathbb{R}^d\times\mathbb{R}^d\to[0,\infty)$ be strictly elliptic, bounded and measurable as in \eqref{kAcond}. Then, for every function $\psi$, measurable in $\Omega$, and admissible in the sense that the closed convex set \[\mathbb{K}_\psi^s=\{v\in H^s_0(\Omega):v\geq\psi \text{ a.e. in }\Omega\}\neq\emptyset,\] there exists a unique $u\in \mathbb{K}_\psi^s$ such that \begin{equation}\label{ObsPb}\mathcal{E}_a(u,v-u)\geq \int_\Omega f_\#(v-u)+\int_{\mathbb{R}^d} \bm{f}\cdot D^s(v-u)\quad\forall v\in \mathbb{K}_\psi^s.\end{equation} 

Moreover, suppose $F,\hat{F}$ are given as in the beginning of this section for two different obstacle problems defined in \eqref{ObsPb}, then the solution map $F\mapsto u$ is Lipschitz continuous, i.e. \[\norm{u-\hat{u}}_{H^s_0(\Omega)}\leq\frac{1}{a_*}\left(C_S\norm{f_\#-\hat{f}_\#}_{L^{2^\#}(\Omega)}+\norm{\bm{f}-\hat{\bm{f}}}_{[L^2(\mathbb{R}^d)]^d}\right).\] 

\end{theorem}

\begin{proof}
This is just a direct application of the Stampacchia theorem, since the bilinear form $\mathcal{E}_a:H^s_0(\Omega)\times H^s_0(\Omega)\to\mathbb{R}$ is bounded and coercive by Theorem \ref{DirichletFormThm}.

For the continuous dependence on data, if $u,\hat{u}$ are the solutions corresponding to different data $F$ and $\hat{F}$ for the obstacle problem respectively, we set $v=\hat{u}$ in the inequality for $u$ and $v=u$ in the inequality for $\hat{u}$, and take the difference to obtain \[\mathcal{E}_a(u-\hat{u},u-\hat{u})\leq\int_\Omega (f_\#-\hat{f}_\#)(u-\hat{u})+\int_{\mathbb{R}^d} (\bm{f}-\hat{\bm{f}})\cdot D^s(u-\hat{u}).\] By the fractional Sobolev inequality and the Cauchy-Schwarz inequality, we have 
\begin{align*}a_*\norm{D^s(u-\hat{u})}_{L^2(\mathbb{R}^d)}^2&\leq\mathcal{E}_a(u-\hat{u},u-\hat{u})\\&\leq\int_\Omega (f_\#-\hat{f}_\#)(u-\hat{u})+\int_{\mathbb{R}^d} (\bm{f}-\hat{\bm{f}})\cdot D^s(u-\hat{u})\\&\leq\norm{f_\#-\hat{f}_\#}_{L^{2^\#}(\Omega)}\norm{u-\hat{u}}_{L^{2^*}(\Omega)}+\norm{\bm{f}-\hat{\bm{f}}}_{[L^2(\mathbb{R}^d)]^d}\norm{D^s(u-\hat{u})}_{[L^2(\mathbb{R}^d)]^d}\\&\leq C_S\left(\norm{f_\#-\hat{f}_\#}_{L^{2^\#}(\Omega)}+\norm{\bm{f}-\hat{\bm{f}}}_{[L^2(\mathbb{R}^d)]^d}\right)\norm{D^s(u-\hat{u})}_{[L^2(\mathbb{R}^d)]^d}\end{align*} where $C_S$ is the Sobolev constant of Lemma \ref{Sobolev}.
\end{proof}

Furthermore, we have the following properties of the solution as in the classical case, making use of the strict T-monotonicity of $\mathcal{E}_a$ and the fractional Poincar\'e inequality. See for example, Chapter IV of \cite{KinderlehrerStampacchia} or Section 4:5--6 of \cite{ObstacleProblems}, where the proofs can be transferred to the nonlocal case almost in the same manner.

\begin{proposition}\label{ObsPbProperties}
\begin{enumerate}[label=(\roman*)]
\item (Comparison principle).
Suppose $u$ is the solution of the variational inequality \eqref{ObsPb} with data $F$ and convex set $\mathbb{K}_\psi^s$, and $\hat{u}$ be the solution with data $\hat{F}$ and convex set $\mathbb{K}^s_{\hat{\psi}}$. If $\psi\geq\hat{\psi}$ and $F\geq\hat{F}$, then $u\geq\hat{u}$ a.e. in $\Omega$.
\item (Weak maximum principle).
In the obstacle problem \eqref{ObsPb}, one has \[u\geq0\text{ a.e. in }\Omega,\text{ if }F\geq0;\text{ and}\]\[u\leq0\vee\sup_\Omega\psi\text{ a.e. in }\Omega,\text{ if }F\leq0.\]
\item (Complementary problem). When $\psi\in H^s_0(\Omega)$, the variational problem \eqref{ObsPb} is equivalent to the nonlinear complementary problem \begin{equation}\label{NonLinComplementary}u\geq\psi,\quad\mathcal{L}_au-F\geq0\quad\text{ and }\quad\langle\mathcal{L}_au-F,u-\psi\rangle=0.\end{equation}
\item (Comparison of supersolutions). 
\begin{enumerate}
    \item If $u$ is the solution to the variational inequality \eqref{ObsPb} and $w$ is any supersolution, then $u\leq w$;
    \item If $v$ and $w$ are two supersolutions to the variational inequality \eqref{ObsPb}, then $v\wedge w$ is also a supersolution,
\end{enumerate}
where a supersolution is an element $w\in H^s_0(\Omega)$ satisfying $w\geq\psi$ and $\mathcal{L}_aw-F\geq0$ in the sense of order dual. 
\item The solution $u$ of the obstacle problem \eqref{ObsPb} is the unique function in $H^s_0(\Omega)$, such that, \[u=\min\{w\in H^s_0(\Omega):\mathcal{L}_aw-F\geq0,w\geq\psi\text{ in }\Omega\}.\]
\item ($L^\infty$ estimates) The following estimate holds: \[\norm{u-\hat{u}}_{L^\infty(\Omega)}\leq\|\psi-\hat{\psi}\|_{L^\infty(\Omega)}.\]
\end{enumerate}
\end{proposition}

Similarly, as in Theorem 6.1 of Chapter 4 of \cite{ObstacleProblems}, we can prove the following additional result for the Dirichlet form $\mathcal{E}_\lambda$ with $\lambda>0$.

\begin{proposition}
All the results in Proposition \ref{ObsPbProperties} holds for $\mathcal{E}_\lambda$ when $\lambda>0$, where \begin{equation}\mathcal{E}_\lambda(u,v)=\mathcal{E}_a(u,v)+\lambda\int_\Omega uv,\quad u,v\in H^s_0(\Omega).\end{equation} Moreover in this case, when $\bm{f}\equiv\bm{0}$, the following maximum principle holds a.e. in $\Omega$: \[0\wedge\inf_\Omega \left(\frac{f_\#}{\lambda}\right)\leq u\leq 0\vee \sup_\Omega \psi\vee\sup_\Omega \left(\frac{f_\#}{\lambda}\right).\]
\end{proposition}

\begin{remark}
If $\Omega=\mathbb{R}^d$ and since the kernel $a$ is defined in the whole $\mathbb{R}^d$, the domain of $\mathcal{E}_\lambda$, $D(\mathcal{E}_\lambda)$, is instead given by $H^s(\mathbb{R}^d)$, and the Dirichlet form $\mathcal{E}_\lambda$ is coercive for $\lambda>0$.
\end{remark}

\subsection{Convergence of the fractional $s$-obstacle problem as $s\nearrow1$}
We start with a continuous dependence property of the Riesz derivatives as $s$ varies.

\begin{lemma}\label{ConvergenceLemma2}
For $u\in H^{s'}_0(\Omega)$, $D^su$ is continuous in $[L^2(\mathbb{R}^d)]^d$ as $s$ varies in $[\sigma,s']$ for $0<\sigma< s'\leq 1$.  As a consequence, we have the following estimate: for $\sigma\leq s\leq1$,  \begin{equation}\label{ConvergenceEq1}\norm{D^\sigma u}_{L^2(\mathbb{R}^d)}\leq c_\sigma\norm{D^su}_{L^2(\mathbb{R}^d)},\end{equation} for any $u\in H^{s}_0(\Omega)$, where the constant $c_\sigma$ is independent of $s$.
\end{lemma}

\begin{proof}
Consider first $u\in C_c^\infty(\Omega)$. Recall that the Fourier transform of $D^su$ is given, by Theorem 1.4 of \cite{SS1}, \[\widehat{D^su}=(2\pi)^si\xi|\xi|^{-1+s}\hat{u},\] where $\hat{u}$ is the Fourier transform of $u$ extended by 0 outside $\Omega$. Since $u\in H^{s'}_0(\Omega)$, the mapping $s\mapsto -(2\pi)^si\xi|\xi|^{-1+s}\hat{u}$ is continuous in $[L^2(\mathbb{R}^d)]^d$ as $s$ varies in $[\sigma,s']$. Therefore, conducting the inverse Fourier transform, we have $\lim_{s\to t}\norm{D^su-D^{t} u}_{L^2(\mathbb{R}^d)}=0$ for $t\in[\sigma,s']$ for $u\in C_c^\infty(\Omega)$. Extending this by density to all $u\in H^{s'}_0(\Omega)$, we have the continuity result on $s$. Finally, the estimate \eqref{ConvergenceEq1} follows similarly to Proposition 2.7 of \cite{BellidoCuetoMoraCorral2021CVPDE} using Lemma \ref{Poincare}.
\end{proof}

As it could be expected, we have a continuous transition from the fractional obstacle problem to the classical local obstacle problem as $s\nearrow1$ in the following sense. We now consider the obstacle problem 

\begin{theorem}\label{ConvergenceThm}
Suppose $\psi$ is such that $\mathbb{K}^1_\psi:=\{v\in H^1_0(\Omega):v\geq\psi \text{ a.e. in }\Omega\}\neq\emptyset$. Let $u^s\in\mathbb{K}^s_\psi$ for $0<s<1$ be the solution to the fractional obstacle problem, i.e. \[\int_{\mathbb{R}^d} AD^su^s\cdot D^s(v-u^s)\geq\langle F,v-u^s\rangle\quad\forall v\in\mathbb{K}^s_\psi,\] where $A:\mathbb{R}^d\to\mathbb{R}^{d\times d}$ is a bounded, measurable and strictly elliptic matrix satisfying \begin{equation}\tag{\ref{Acoer}}a_* |z|^2\leq A(x)z\cdot z \text{ and } A(x)z\cdot z^*\leq a^* |z||z^*|,\end{equation} and $F\in H^{-\sigma}(\Omega)$. Then, there exists a unique solution $u^s\in H^s_0(\Omega)$. Furthermore, the sequence $(u^s)_s$ converges strongly to $u$ in $H^\sigma_0(\Omega)$ as $s\nearrow1$ for any fixed $0<\sigma<1$, where $u\in \mathbb{K}^1_\psi$ solves uniquely the obstacle problem for $s=1$, i.e. \[\int_\Omega ADu\cdot D(v-u)\geq\langle F,v-u\rangle\quad\forall v\in\mathbb{K}^1_\psi.\]
\end{theorem}

\begin{proof}

The existence of a solution follows from a direct application of the Stampacchia theorem, since the bilinear form \[\langle\tilde{\mathcal{L}}_Au,v\rangle=\int_{\mathbb{R}^d}AD^su\cdot D^sv\] is bounded \[\langle\tilde{\mathcal{L}}_Au,v\rangle\leq a^*\norm{u}_{H^s_0(\Omega)}\norm{v}_{H^s_0(\Omega)}\quad\forall u,v\in H^s_0(\Omega)\] and coercive \[\langle\tilde{\mathcal{L}}_Au,v\rangle\geq a_*\norm{D^su}^2_{L^2(\mathbb{R}^d)}=a_*\norm{u}^2_{H^s_0(\Omega)}.\]

For uniqueness, if $u,\hat{u}$ are the solutions corresponding to the same data $F$ for the obstacle problem, we set $v=\hat{u}$ in the inequality for $u$ and $v=u$ in the inequality for $\hat{u}$, and take the difference to obtain \[a_*\norm{D^s(u-\hat{u})}_{L^2(\mathbb{R}^d)}^2\leq\int_{\mathbb{R}^d}AD^s(u-\hat{u})\cdot D^s(u-\hat{u})\leq\int_\Omega (f_\#-f_\#)(u-\hat{u})+\int_{\mathbb{R}^d} (\bm{f}-\bm{f})\cdot D^s(u-\hat{u})=0,\] so $u\equiv\hat{u}$.

Next, we want to show the convergence of the fractional obstacle problem to the classical one. We first prove an a priori estimate for $D^su^s$. For $v^1\in\mathbb{K}^1_\psi=\bigcap_{\sigma\leq s<1}\mathbb{K}^s_\psi$, by Cauchy-Schwarz inequality and Sobolev inequality,
\begin{align*}a_* \norm{D^su^s}_{L^2(\mathbb{R}^d)}^2\leq&\int_{\mathbb{R}^d} AD^su^s\cdot D^su^s\\\leq&\int_{\mathbb{R}^d} AD^su^s\cdot D^sv^1-\langle F,v^1-u^s\rangle\\\leq&\frac{{a^*}^2\epsilon}{2}\norm{D^su^s}_{L^2(\mathbb{R}^d)}^2+\frac{1}{2\epsilon}\norm{D^sv^1}_{L^2(\mathbb{R}^d)}^2-\langle F,v^1\rangle+\frac{1}{2\epsilon'}\norm{F}_{H^{-s}(\Omega)}^2+\frac{\epsilon'}{2}\norm{D^su^s}_{L^2(\mathbb{R}^d)}^2\\\leq&\frac{a_*}{4}\norm{D^su^s}_{L^2(\mathbb{R}^d)}^2+\frac{{a^*}^2}{a_*}\norm{D^sv^1}_{L^2(\mathbb{R}^d)}^2-\langle F,v^1\rangle+\frac{c_\sigma^2}{a_*}\norm{F}_{H^{-\sigma}(\Omega)}^2+\frac{a_*}{4}\norm{D^su^s}_{L^2(\mathbb{R}^d)}^2\end{align*} by taking $\epsilon=\frac{a_*}{2{a^*}^2}$ and $\epsilon'=\frac{a_*}{2}$ and $c_\sigma^2$ may be chosen independent of $s$ for $0<s\leq1$, as a consequence of \eqref{ConvergenceEq1} for the dual norms $\norm{\cdot}_{H^{-s}(\Omega)}$. Therefore, we have \[\norm{D^su^s}_{L^2(\mathbb{R}^d)}\leq C,\] where the constant  $C=C(\sigma,a_*,a^*)>0$ is independent of $s\geq\sigma$.

Also by \eqref{ConvergenceEq1}, for $\sigma\leq s<1$, we have \begin{equation}\norm{D^\sigma u^s}_{L^2(\mathbb{R}^d)}\leq c_\sigma\norm{D^su^s}_{L^2(\mathbb{R}^d)}\leq C,\end{equation} for some constant $C$ independent of $s$, $\sigma\leq s<1$, and we may take a sequence \[D^su^s\xrightharpoonup[s\nearrow1]{}\zeta\quad \text{ in }[L^2(\mathbb{R}^d)]^d\text{-weak}\] for some $\zeta$. By compactness, since $u^s$ is also uniformly bounded in $H^\sigma_0(\Omega)$, there exists a subsequence and a limit $u\in L^2(\Omega)$ such that \[u^s\xrightarrow[s\nearrow1]{}u\quad\text{ strongly in }L^2(\Omega).\]

Now, by Lemma \ref{ConvergenceLemma2}, for all $\Phi\in [C_c^\infty(\Omega)]^d$, denoting by $\tilde{\Phi}$ the zero extension of $\Phi$ outside $\Omega$, \[D^s\cdot\tilde{\Phi}\to D\cdot\tilde{\Phi}\quad\text{ in }[L^2(\mathbb{R}^d)]^d,\] therefore \[\int_{\mathbb{R}^d} D^su^s\cdot\tilde{\Phi}=-\int_{\mathbb{R}^d} \tilde{u}^s(D^s\cdot\tilde{\Phi})\xrightarrow[s\nearrow1]{}-\int_{\mathbb{R}^d} \tilde{u}(D\cdot\tilde{\Phi}).\]  But by the a priori estimate on $D^su^s$, \[\left|\int_{\mathbb{R}^d} D^su^s\cdot\tilde{\Phi}\right|\leq C\norm{\Phi}_{[L^2(\Omega)]^d},\] which implies that \[\left|\int_{\mathbb{R}^d} \tilde{u}(D\cdot\tilde{\Phi})\right|\leq C\norm{\Phi}_{[L^2(\Omega)]^d}.\] This means that $D\tilde{u}\in[L^2(\mathbb{R}^d)]^d$, and since $\Omega$ has a Lipschitz boundary, $\widetilde{Du}=D\tilde{u}$, so $Du\in[L^2(\Omega)]^d$. Together with the first inequality in \eqref{ConvergenceEq1} which implies that $u\in \mathbb{K}^\sigma_\psi$ for any $\sigma<1$, we have $u\in\mathbb{K}^1_\psi$.

Furthermore, by Lemma \ref{ConvergenceLemma2}, $D^su\to \widetilde{Du}$ strongly in $[L^2(\mathbb{R}^d)]^d$ as $s\nearrow1$, so \[\int_{\mathbb{R}^d} D^s(u^s-u)\cdot\tilde{\Phi}=-\int_{\mathbb{R}^d}(\widetilde{u^s-u})(D^s\cdot\Phi)\to0,\] therefore \[\zeta=\lim_{s\nearrow1}D^su^s=\widetilde{Du}.\]

Finally, it remains to show that $u$ satisfies the obstacle problem for $s=1$. For any $v\in \mathbb{K}^1_\psi\subset\mathbb{K}^s_\psi$, since $D^su^s$ are uniformly bounded, we have, up to a subsequence and using the lower-semicontinuity of $A_{sym}=\frac{1}{2}(A+A^T)$, \begin{align*}\int_\Omega ADu\cdot Dv&=\int_{\mathbb{R}^d} A\widetilde{Du}\cdot \widetilde{Dv}\\&=\lim_{s\nearrow1}\int_{\mathbb{R}^d} AD^su^s\cdot D^sv\\&\geq\lim_{s\nearrow1}\langle F,v-u^s\rangle+\liminf_{s\nearrow1}\int_{\mathbb{R}^d} AD^su^s\cdot D^su^s\\&=\langle F,v-u\rangle+\liminf_{s\nearrow1}\int_{\mathbb{R}^d} A_{sym}D^su^s\cdot D^su^s\\&\geq\langle F,v-u\rangle+\int_{\mathbb{R}^d} A_{sym}\widetilde{Du}\cdot \widetilde{Du}\\&=\langle F,v-u\rangle+\int_\Omega ADu\cdot Du\end{align*} since $D^su^s\rightharpoonup \widetilde{Du}$ weakly in $[L^2(\mathbb{R}^d)]^d$ and $D^sv\to \widetilde{Dv}$ strongly in $[L^2(\mathbb{R}^d)]^d$. The conclusion follows by the compactness of the inclusion of $H^\sigma_0(\Omega)$ in $H^{\sigma'}_0(\Omega)$ when $\sigma>\sigma'$.

\end{proof}

\begin{remark}
The case with $A=\mathbb{I}$ corresponds to the obstacle problem for the fractional Laplacian and was first considered by Silvestre in \cite{SilvestreThesis}. Indeed, from \eqref{EquivNorms}, since \[\int_{\mathbb{R}^d} D^s u\cdot D^s v= \frac{c_{d,s}^2}{2}\int_{\mathbb{R}^d}\int_{\mathbb{R}^d}\frac{(u(x)-u(y))(v(x)-v(y))}{|x-y|^{d+2s}}\,dx\,dy,\] also holds for $u,v\in H^s_0(\Omega)$, Theorem \ref{ConvergenceThm} gives the convergence of the solution $u^s$ to the nonlocal obstacle problem \eqref{ObsPb}, which is equivalent, up to a constant, to \[u^s\in\mathbb{K}^s_\psi:\int_{\mathbb{R}^d} D^su^s\cdot D^s(v-u^s)\geq\langle F,v-u^s\rangle\quad\forall v\in\mathbb{K}^s_\psi\] towards the solution $u$ of the classical problem \[u\in\mathbb{K}^1_\psi:\int_\Omega Du\cdot D(v-u)\geq\langle F,v-u\rangle\quad\forall v\in\mathbb{K}^1_\psi.\]
\end{remark}

\begin{remark}
A similar convergence result as $s\nearrow1$ remains open for the general nonlocal operator $\mathcal{L}_a$.
\end{remark}

\section{Lewy-Stampacchia Inequalities and Local Regularity}

In this section, we take $f=f_\#$ and $\bm{f}=0$. We give a direct proof of the Lewy-Stampacchia inequalities. This will follow much of the approach of Section 5:3.3 in \cite{ObstacleProblems} or Chapter IV of \cite{KinderlehrerStampacchia}. The Lewy-Stampacchia inequalities will allow us to apply the results of \cite{KassmannDyda}\cite{FelsingerKassmannVoigt} to obtain local H\"older regularity of the solutions when $a$ is symmetric, and additional regularity on fractional Sobolev spaces when $\mathcal{L}_a=(-\Delta)^s$ using \cite{BiccariWarmaZuazua}.

\subsection{Bounded penalization of the obstacle problem in $H^s_0(\Omega)$}
Assume now that the obstacle $\psi\in H^s(\mathbb{R}^d)$, so that we may define $\mathcal{L}_a\psi\in H^{-s}(\Omega)$ by \eqref{LAdef} for any test function $v\in H^s_0(\Omega)$, and $\psi$ is such that the convex set $\mathbb{K}_\psi^s\neq\emptyset$. Consider the approximation to the obstacle problem, where the penalization is based on any nondecreasing Lipschitz function $\theta:\mathbb{R}\to[0,1]$ such that \[\theta\in C^{0,1}(\mathbb{R}),\quad\theta'\geq0,\quad\theta(+\infty)=1\quad\text{ and }\theta(t)=0\text{ for }t\leq0;\]\[\exists C_\theta>0:[1-\theta(t)]t\leq C_\theta,\quad t>0.\] Then, for any $\varepsilon>0$, consider the family of functions $\theta_\varepsilon(t)=\theta\left(\frac{t}{\varepsilon}\right),\quad t\in\mathbb{R}$, which converges as $\varepsilon\to0$ to the multi-valued Heaviside graph. Examples of such sequences of functions include $\theta(t)=t/(1+t)$, $\theta(t)=(2/\pi)\arctan t$, or from any non-decreasing Lipschitz function $0\leq\theta\leq1$ such that $\theta(t)=1$ for $t\geq t_*>0$.

Assume that \begin{equation}\label{fAssump}f,(\mathcal{L}_a\psi-f)^+\in L^{2^\#}(\Omega).\end{equation} For $\zeta\in L^{2^\#}(\Omega)$ such that \[\zeta\geq(\mathcal{L}_a\psi-f)^+\text{ a.e. in }\Omega,\] consider now the one parameter family of approximating semilinear problems in variational form
\[u_\varepsilon\in H^s_0(\Omega):\quad\mathcal{E}_a(u_\varepsilon,v)+\int_\Omega\zeta\theta_\varepsilon(u_\varepsilon-\psi)v=\int_\Omega (f+\zeta)v\quad\forall v\in H^s_0(\Omega).\tag{\ref{ObsPb}$_\varepsilon$}\] Arguing as in the proof of Theorem 5:3.1 of \cite{ObstacleProblems}, we have the following theorem.

\begin{theorem}\label{PenalizedProbSoln}
The unique solution $u_\varepsilon$ of the semilinear boundary value problem (\ref{ObsPb}$_\varepsilon$) is such that $u_\varepsilon\in\mathbb{K}^s_\psi$ for each $\varepsilon>0$, and it defines a monotone decreasing sequence, converging as $\varepsilon\to0$ to the solution $u$ of the obstacle problem \eqref{ObsPb} with the error estimate \begin{equation}\label{PenalizationErrorEst}\norm{u-u_\varepsilon}_{H^s_0(\Omega)}^2\leq\varepsilon(C_\theta/a_*)\norm{\zeta}_{L^1(\Omega)}.\end{equation}
\end{theorem}

\begin{remark} 
\begin{enumerate}[label=(\roman*)]
    \item We can formally interpret the variational equation (\ref{ObsPb}$_\varepsilon$) as corresponding to the following semilinear boundary value problem:  \[\mathcal{L}_au_\varepsilon+\zeta\theta_\varepsilon(u_\varepsilon-\psi)=f+\zeta\text{ in }\Omega, \quad u_\varepsilon=0 \text{ on }\mathbb{R}^d\backslash\Omega.\]
    \item One can also consider the translated penalization, given by \[\bar{\theta}_\varepsilon(t)=1-\theta\left(-\frac{t}{\varepsilon}\right)=\bar{\theta}\left(-\frac{t}{\varepsilon}\right)\] for $t\in\mathbb{R}$ and $\varepsilon>0$, to approach the solution of the obstacle problem from below using monotonicity. Then we have that the  unique solution $\bar{u}_\varepsilon$ of the penalized problem \[\bar{u}_\varepsilon\in H^s_0(\Omega):\quad\mathcal{E}_a(\bar{u}_\varepsilon,v)+\int_\Omega\zeta\bar{\theta}_\varepsilon(\bar{u}_\varepsilon-\psi)v=\int_\Omega (f+\zeta)v\quad\forall v\in H^s_0(\Omega).\tag{\ref{ObsPb}$_\varepsilon^\sim$}\] defines a monotone increasing sequence converging to the solution $u$ of the obstacle problem \eqref{ObsPb} weakly in $H^s_0(\Omega)$.
    \item For special choices of the function $\theta$, we can estimate the uniform convergence for the approximations of $u$ by penalization. Suppose that, in addition to the conditions on $\theta$, \[\theta(t)=1\text{ for }t\geq1,\] then the approximating solution $u_\varepsilon$ of (\ref{ObsPb}$_\varepsilon$) verifies, for each $\varepsilon>0$, \[u_\varepsilon-\varepsilon\leq u\leq u_\varepsilon\text{ a.e. in }\Omega.\] If, additionally, $\theta$ verifies \[t\leq\theta(t)\leq1\text{ if }0\leq t\leq1,\] the approximating solution $\bar{u}_\varepsilon$ yields \[\bar{u}_\varepsilon\leq u\leq u_\varepsilon\text{ and }0\leq u_\varepsilon-\bar{u}_\varepsilon\leq\varepsilon\text{ a.e. in }\Omega.\]
\end{enumerate}

\end{remark}

From this theorem, we can derive the Lewy-Stampacchia inequality.

\begin{theorem}[Lewy-Stampacchia inequality]\label{LewyStampacchia}
Under the assumptions \eqref{fAssump}, the solution $u$ of the obstacle problem satisfies \[f\leq \mathcal{L}_au\leq f\vee \mathcal{L}_a\psi\quad\text{ a.e. in }\Omega.\] In particular,  $\mathcal{L}_au\in L^{2^\#}(\Omega)$.
\end{theorem}

\begin{proof}
Choosing $\zeta=(\mathcal{L}_a\psi-f)^+$ in (\ref{ObsPb}$_\varepsilon$), and making use of the property of $\theta$ that $0\leq1-\theta_\varepsilon\leq1$, then for any $\varepsilon>0$ and any $v\in H^s_0(\Omega)$, $v\geq0$, we have \[\int_\Omega\mathcal{L}_au_\varepsilon v = \int_\Omega  [f+(\mathcal{L}_a\psi-f)^+(1-\theta_\varepsilon)]v \leq \int_\Omega [f+(\mathcal{L}_a\psi-f)^+]v\] from the variational form, and on the other hand \[\mathcal{L}_au_\varepsilon=f+\zeta-\zeta\theta_\varepsilon=f+\zeta(1-\theta_\varepsilon)\geq f\] holds a.e. in $\Omega$. Together, these give \[\int_\Omega f v \leq \int_\Omega \mathcal{L}_au_\varepsilon v \leq \int_\Omega [f+(\mathcal{L}_a\psi-f)^+] v =\int_\Omega [f\vee(\mathcal{L}_a\psi)]v.\] Letting $\varepsilon\to0$, this holds for $u$. Since $v$ is arbitrary, we have the result.
\end{proof}

\begin{remark}
The Lewy-Stampacchia inequalities for nonlocal obstacle problems have been first obtained in \cite{ServadeiValdinoci2013RMILewyStampacchia} for a class of symmetric integrodifferential operators $\mathcal{L}_K$, with even kernels $K$, which are also strictly T-monotone and include the fractional Laplacian, and with $f$ and $\mathcal{L}_K\psi\in L^\infty(\Omega)$.
\end{remark}

\subsection{Multiple obstacles problem}
\subsubsection{Two obstacles problem}
We next consider the two obstacles problem with a Dirichlet boundary condition in a bounded domain $\Omega\subset\mathbb{R}^d$ with Lipschitz boundary, which consists of finding $u\in\mathbb{K}^s_{\psi,\varphi}$ such that \begin{equation}\label{2Obstacles}\mathcal{E}_a(u,v-u)\geq\int_\Omega f(v-u),\quad\forall v\in \mathbb{K}^s_{\psi,\varphi},\end{equation} where $f\in L^{2^\#}(\Omega)$ and \begin{equation}\label{2ObstaclesConvex}\mathbb{K}^s_{\psi,\varphi}=\{v\in H^s_0(\Omega):\psi\leq v\leq\varphi\text{ a.e. in }\Omega\}.\end{equation} Assume that $\varphi$ and $\psi$ are measurable and admissible obstacles in $\Omega$ such that $\mathbb{K}^s_{\psi,\varphi}\neq\emptyset$. When $\varphi,\psi\in H^s(\mathbb{R}^d)$, a sufficient condition for these two assumptions to hold is to assume $\varphi\geq\psi$ a.e. in $\Omega$ and $\varphi\geq0\geq\psi$ a.e. in $\Omega^c$.

\begin{theorem} The two obstacles problem has a unique solution. Moreover, if $u=u(f,\varphi,\psi)$ and $\hat{u}=u(\hat{f},\hat{\varphi},\hat{\psi})$ are solutions in $\mathbb{K}^s_{\psi,\varphi}$ and in $\mathbb{K}^s_{\hat{\psi},\hat{\varphi}}$, respectively, of the two obstacles problem, then \[f\geq\hat{f},\varphi\geq\hat{\varphi},\psi\geq\hat{\psi}\text{ implies }u\geq\hat{u}\text{ a.e. in }\Omega.\] In addition, if $f=\hat{f}$, we have the $L^\infty$ estimate \[\norm{u-\hat{u}}_{L^\infty(\Omega)}\leq\|\psi-\hat{\psi}\|_{L^\infty(\Omega)}+\norm{\varphi-\hat{\varphi}}_{L^\infty(\Omega)}.\]
\end{theorem}
\begin{proof}
The existence and uniqueness follows as, in the previous sections, from the monotonicity, coercivity, continuity and boundedness of the operator $\mathcal{L}_a$ and the Stampacchia theorem. The comparison property follows also as previous by the T-monotonicity of $\mathcal{L}_a$.

The $L^\infty$ estimate follows as well, as in the classical one obstacle problem.
\end{proof}

Corresponding to the two obstacles problem, we also have the Lewy-Stampacchia inequality.

\begin{theorem}\label{LewyStampacchia2}
The solution $u$ of the two obstacles problem, for $f,\mathcal{L}_a\varphi,\mathcal{L}_a\psi\in L^{2^\#}(\Omega)$ such that $\varphi,\psi\in H^s(\mathbb{R}^d)$ are compatible and $\mathcal{L}_a\varphi,\mathcal{L}_a\psi$ are given by \eqref{LAdef}, satisfies \begin{equation}\label{LewyStampacchia2eq}f\wedge\mathcal{L}_a\varphi\leq \mathcal{L}_au\leq f\vee\mathcal{L}_a\psi\quad\text{ a.e. in }\Omega,\end{equation} and therefore $\mathcal{L}_au\in L^{2^\#}(\Omega)$.
\end{theorem}

\begin{proof}
The proof is similar to that of the classical case $s=1$, now for two obstacles. Consider the penalized problem given by \[u_\varepsilon\in H^s_0(\Omega):\quad\langle \mathcal{L}_au_\varepsilon,v\rangle+\int_\Omega\zeta_\psi\theta_\varepsilon(u_\varepsilon-\psi)v-\int_\Omega\zeta_\varphi\theta_\varepsilon(\varphi-u_\varepsilon)v=\int_\Omega (f+\zeta_\psi-\zeta_\varphi)v\quad\forall v\in H^s_0(\Omega)\] for \[\zeta_\psi\geq(\mathcal{L}_a\psi-f)^+,\quad\zeta_\varphi\geq(\mathcal{L}_a\varphi-f)^-,\] with $\theta_\varepsilon(t)=1$ for $t\geq\varepsilon$. Then, there is a unique solution $u_\varepsilon\in H^s_0(\Omega)$ such that $\psi\leq u_\varepsilon\leq\varphi+\varepsilon$ for each $\varepsilon>0$. Indeed, we obtain the existence and uniqueness of the solution by the Stampacchia theorem as before. To show that $\psi\leq u_\varepsilon\leq\varphi+\varepsilon$, we have \[\langle\mathcal{L}_a\psi,v\rangle\leq\langle(\mathcal{L}_a\psi-f)^++f,v\rangle\leq\int_\Omega(\zeta_\psi+f)v\quad\forall v\in H^s_0(\Omega),v\geq0,
\quad\text{ and }\]\[ \langle\mathcal{L}_a\varphi,v\rangle=\langle (\mathcal{L}_a\varphi-f)+f,v\rangle\geq\langle f-(\mathcal{L}_a\varphi-f)^-,v\rangle\geq\int_\Omega (f-\zeta_\varphi)v\quad\forall v\in H^s_0(\Omega),v\geq0.\] Taking $v=(\psi-u_\varepsilon)^+\in H^s_0(\Omega)$ and using the strict T-monotonicity of $\mathcal{L}_a$, we have \begin{align*} a_*\norm{(\psi-u_\varepsilon)^+}_{H^s_0(\Omega)}^2\leq&\mathcal{E}_a(\psi,(\psi-u_\varepsilon)^+)-\mathcal{E}_a(u_\varepsilon,(\psi-u_\varepsilon)^+)\\\leq&\int_\Omega (f+\zeta_\psi)(\psi-u_\varepsilon)^+-\int_\Omega \{f+\zeta_\psi[1-\theta_\varepsilon(u_\varepsilon-\psi)]-\zeta_\varphi[1-\theta_\varepsilon(\varphi-u_\varepsilon)]\}(\psi-u_\varepsilon)^+\\=&\int_\Omega\zeta_\psi\theta_\varepsilon(u_\varepsilon-\psi)(\psi-u_\varepsilon)^+-\int_\Omega \{\zeta_\varphi[1-\theta_\varepsilon(\varphi-u_\varepsilon)]\}(\psi-u_\varepsilon)^+\\\leq&0\end{align*} because the first term is non-positive, while the factors in the second term are all non-negative. Therefore, $u_\varepsilon\geq\psi$. Similarly, taking $v=(u_\varepsilon-\varphi-\varepsilon)^+\in H^s_0(\Omega)$ gives 
\begin{align*}a_*\norm{(u_\varepsilon-\varphi-\varepsilon)^+}_{H^s_0(\Omega)}^2=&\,a_*\norm{D^s(u_\varepsilon-\varphi-\varepsilon)^+}_{L^2(\mathbb{R}^d)}^2\\\leq&\,\mathcal{E}_a((u_\varepsilon-\varphi-\varepsilon)^+,(u_\varepsilon-\varphi-\varepsilon)^+)\\\leq&\,\mathcal{E}_a(u_\varepsilon-\varphi-\varepsilon,(u_\varepsilon-\varphi-\varepsilon)^+)\\\leq&\,\mathcal{E}_a(u_\varepsilon-\varphi,(u_\varepsilon-\varphi-\varepsilon)^+)\\=&\,\left[\mathcal{E}_a(u_\varepsilon,(u_\varepsilon-\varphi-\varepsilon)^+)-\mathcal{E}_a(\varphi,(u_\varepsilon-\varphi-\varepsilon)^+)\right]\\\leq&\,\int_\Omega \{f+\zeta_\psi[1-\theta_\varepsilon(u_\varepsilon-\psi)]-\zeta_\varphi[1-\theta_\varepsilon(\varphi-u_\varepsilon)]\}(u_\varepsilon-\varphi-\varepsilon)^+\\&\,-\int_\Omega (f-\zeta_\varphi)(u_\varepsilon-\varphi-\varepsilon)^+\\=&\,\int_\Omega\{\zeta_\psi[1-\theta_\varepsilon(u_\varepsilon-\psi)]\}(u_\varepsilon-\varphi-\varepsilon)^++\int_\Omega [\zeta_\varphi\theta_\varepsilon(\varphi-u_\varepsilon)](u_\varepsilon-\varphi-\varepsilon)^+\\=&\,\int_\Omega\zeta_\psi[1-\theta_\varepsilon(u_\varepsilon-\psi)](u_\varepsilon-\varphi-\varepsilon)^+\\\leq&\,\int_\Omega\zeta_\psi[1-\theta_\varepsilon(u_\varepsilon-\psi)](u_\varepsilon-\psi-\varepsilon)^+\text{ since }\varphi\geq\psi\\=&\,0
\end{align*}Therefore, $u_\varepsilon\leq\varphi+\varepsilon$.

Now, we can show that $u_\varepsilon\to u$, so $\psi\leq u_\varepsilon\leq\varphi+\varepsilon$ converges to $\psi\leq u\leq\varphi$. Take $v=w-u_\varepsilon$ in the penalized problem above for arbitrary $w\in\mathbb{K}^s_{\psi,\varphi}$, then \begin{align*}\mathcal{E}_a(u_\varepsilon,w-u_\varepsilon)=&\int_\Omega f(w-u_\varepsilon)+\int_\Omega\zeta_\psi[1-\theta_\varepsilon(u_\varepsilon-\psi)](w-u_\varepsilon)-\int_\Omega\zeta_\varphi[1-\theta_\varepsilon(\varphi-u_\varepsilon)](w-u_\varepsilon)\\\geq&\int_\Omega f(w-u_\varepsilon)+\int_\Omega\zeta_\psi[1-\theta_\varepsilon(u_\varepsilon-\psi)](\psi-u_\varepsilon)-\int_\Omega\zeta_\varphi[1-\theta_\varepsilon(\varphi-u_\varepsilon)](\varphi-u_\varepsilon)\\=&\int_\Omega f(w-u_\varepsilon)-\varepsilon\int_\Omega\zeta_\psi[1-\theta_\varepsilon(u_\varepsilon-\psi)]\frac{u_\varepsilon-\psi}{\varepsilon}-\varepsilon\int_\Omega\zeta_\varphi[1-\theta_\varepsilon(\varphi-u_\varepsilon)]\frac{\varphi-u_\varepsilon}{\varepsilon}\\\geq&\int_\Omega f(w-u_\varepsilon)-\varepsilon C_\theta\int_\Omega(\zeta_\psi+\zeta_\varphi).\end{align*} Now, taking $w=u\in\mathbb{K}^s_{\psi,\varphi}$, we obtain \[\mathcal{E}_a(u_\varepsilon,u-u_\varepsilon)\geq\int_\Omega f(u-u_\varepsilon)-\varepsilon C_\theta\int_\Omega(\zeta_\psi+\zeta_\varphi),\] but taking $v=u_\varepsilon\in\mathbb{K}^s_\psi$ in the original obstacle problem \eqref{ObsPb}, we have \[\mathcal{E}_a(u,u_\varepsilon-u)\geq \int_\Omega  f(u_\varepsilon-u).\] Taking the difference of these two equations, by the linearity of $\mathcal{E}_a$, we have \[\mathcal{E}_a(u_\varepsilon-u,u_\varepsilon-u)\leq\varepsilon C_\theta\int_\Omega(\zeta_\psi+\zeta_\varphi).\] Using the ellipticity of $a$, we have \[\varepsilon C_\theta\int_\Omega(\zeta_\psi+\zeta_\varphi)\geq\mathcal{E}_a(u_\varepsilon-u,u_\varepsilon-u)\geq a_*\norm{D^s(u_\varepsilon-u)}_{L^2({\mathbb{R}^d})}^2=a_*\norm{u_\varepsilon-u}_{H^s_0(\Omega)}^2.\] Therefore, $u_\varepsilon\to u$ in $H^s_0(\Omega)$.

Choosing $\zeta_\psi=(\mathcal{L}_a\psi-f)^+$ and $\zeta_\varphi=(\mathcal{L}_a\varphi-f)^-$, and making use of the property of $\theta$ that $0\leq1-\theta_\varepsilon\leq1$, then for any $\varepsilon>0$ and any $v\in H^s_0(\Omega)$, $v\geq0$, we have \[\int_\Omega(\mathcal{L}_au_\varepsilon)v=\int_\Omega[ f+(\mathcal{L}_a\psi-f)^+(1-\theta_\varepsilon)-(\mathcal{L}_a\varphi-f)^-(1-\theta_\varepsilon)]v\leq\int_\Omega[f+(\mathcal{L}_a\psi-f)^+]v\] and \[\int_\Omega(\mathcal{L}_au_\varepsilon)v=\int_\Omega[f+(\mathcal{L}_a\psi-f)^+(1-\theta_\varepsilon)-(\mathcal{L}_a\varphi-f)^-(1-\theta_\varepsilon)]v\geq\int_\Omega[ f-(\mathcal{L}_a\varphi-f)^-]v.\] Therefore, \[\int_\Omega[f\wedge(\mathcal{L}_a\varphi)]v=\int_\Omega[f-(\mathcal{L}_a\varphi-f)^-]v\leq\int_\Omega(\mathcal{L}_au_\varepsilon)v\leq\int_\Omega[f+(\mathcal{L}_a\psi-f)^+]v=\int_\Omega[f\vee(\mathcal{L}_a\psi)]v.\] Letting $\varepsilon\to0$, this holds for $u$. Since $v$ is arbitrary, we conclude \eqref{LewyStampacchia2eq}.
\end{proof}

\subsubsection{N-membranes problem}

We consider now the N-membranes problem, which consists of: To find $u=(u_1,u_2,\dots,u_N)\in\mathbb{K}_N^s$ satisfying \begin{equation}\label{NObstacles}\sum_{i=1}^N\mathcal{E}_a(u_i,v_i-u_i)\geq\sum_{i=1}^N\int_\Omega f^i(v_i-u_i),\quad\forall (v_1,\dots,v_N)\in \mathbb{K}^s_N,\end{equation} where $\mathbb{K}^s_N$ is the convex subset of $[H^s_0(\Omega)]^N$ defined by \begin{equation}\label{NObstaclesConvex}\mathbb{K}^s_N=\{(v_1,\dots,v_N)\in [H^s_0(\Omega)]^N:v_1\geq\dots\geq v_N\text{ a.e. in }\Omega\}\end{equation} and $f^i,\dots,f^N\in L^{2^\#}(\Omega)$. As in the previous sections, the existence and uniqueness follows easily. Furthermore, the following Lewy-Stampacchia type inequality also holds.

\begin{theorem}\label{LewyStampacchiaN}
The solution $u=(u_1,\dots u_N)$ of the N-membranes problem satisfies a.e. in $\Omega$ \begin{align*}f^1\wedge \mathcal{L}_au_1&\leq f^1\vee\dots\vee f^N\\f^1\wedge f^2\leq \mathcal{L}_au_2&\leq f^2\vee\dots\vee f^N\\\vdots&\\f^1\wedge\dots\wedge f^{N-1}\leq \mathcal{L}_au_{N-1}&\leq f^{N-1}\vee f^N\\f^1\wedge\dots\wedge f^N\leq \mathcal{L}_au_N&\leq f^N,\end{align*} and $\mathcal{L}_au\in [L^{2^\#}(\Omega)]^N$.
\end{theorem}

\begin{proof}
Choosing $(v,u_2,\dots,u_N)\in\mathbb{K}_N^s$ with $v\in\mathbb{K}_{u_2}^s$, we see that $u_1\in\mathbb{K}_{u_2}^s$ solves \eqref{ObsPb} with $f=f^1$ and by Theorem \ref{LewyStampacchia} \[f^1\leq\mathcal{L}_au_1\leq f^1\vee\mathcal{L}_au_2\text{ a.e. in }\Omega.\] Analogously, $u_j\in\mathbb{K}^s_{u_{j+1},u_{j-1}}$ solves the two obstacles problem with $f=f^j$, $j=2,3,\dots,N-1$, and satisfies, by \eqref{LewyStampacchia2eq}, \[f^j\wedge\mathcal{L}_au_{j-1}\leq \mathcal{L}_au_j\leq f^j\vee\mathcal{L}_au_{j+1}\quad\text{ a.e. in }\Omega.\] Finally,  $u_N$ solves the one obstacle problem with an upper obstacle $\varphi= u_{N-1}$, and so by the symmetric Lewy-Stampacchia estimates given in Theorem \ref{LewyStampacchia}, we have \[f^N\wedge\mathcal{L}_au_{N-1}\leq \mathcal{L}_au_N\leq f^N\quad\text{ a.e. in }\Omega.\] The proof concludes by simple iteration (see Theorem 5.1 of \cite{RodriguesTeymurazyan}).

\end{proof}

\begin{remark}
The solution can also approximated by the bounded penalization given in \cite{Rodrigues2005NMembrane} by \[u_\varepsilon\in H^s_0(\Omega):\quad\langle \mathcal{L}_au^\varepsilon_i,v_i\rangle+\int_\Omega\zeta_i\theta_\varepsilon(u^\varepsilon_i-u^\varepsilon_{i+1})v_i-\int_\Omega\zeta_{i-1}\theta_\varepsilon(u^\varepsilon_{i-1}-u^\varepsilon_i)v_i=\int_\Omega (f^i+\zeta_i-\zeta_{i-1})v_i\quad\forall v_i\in H^s_0(\Omega)\] for \begin{align*}\zeta_0=&\,\max\left\{\frac{f^1+\cdots+f^i}{i}:i=1,\dots,N\right\},\\\zeta_i=&\,i\zeta_0-(f^1+\cdots+f^i)\quad\text{ for }i=1,\dots,N\end{align*} with $\theta_\varepsilon(t)=1$ for $t\geq\varepsilon$ and $u^\varepsilon_0=+\infty,u^\varepsilon_{N+1}=-\infty.$

As in \cite{Rodrigues2005NMembrane}, we then have the strong convergence of $u^\varepsilon\to u$ in $[H^s_0(\Omega)]^N$, thereby giving the Lewy-Stampacchia inequalities in Theorem \ref{LewyStampacchiaN}.
\end{remark}

\subsection{Local regularity of solutions}
We make use of the Lewy-Stampacchia inequalities to show local regularity for the three types of nonlocal obstacle problems, but first it is useful to obtain the global boundedness of the solutions.

Let $s\in(0,1)$. Suppose that 
\begin{enumerate}[label=(\alph*)]
    \item $f,\mathcal{L}_a\psi\in L^p(\Omega)$ for some $p>\frac{d}{2s}$ for the one obstacle problem,
    \item $f\wedge\mathcal{L}_a\varphi$ and $f\vee\mathcal{L}_a\psi$ are in $L^p(\Omega)$ for some $p>\frac{d}{2s}$ for the two obstacles problem, or
    \item $f^i\in L^p(\Omega)$ for $i=1,\dots,N$ for some $p>\frac{d}{2s}$ for the N-membranes problem.
\end{enumerate}

\begin{theorem} Let $u$ denote the solutions of the one obstacle problem \eqref{ObsPb}, or the two obstacles problem \eqref{2Obstacles}, or $u=u_i$ for $i=1,\dots,N$ of the N-membranes problem \eqref{NObstacles}, respectively, under the assumptions (a), (b) or (c) above. Then $g=\mathcal{L}_au\in L^p(\Omega)$, with $p>\frac{d}{2s}$ and there exists a constant $C$, depending only on $a_*$, $a^*$, $d$, $\Omega$, $\norm{u}_{H^s_0(\Omega)}$, $\norm{g}_{L^p(\Omega)}$ and $s$, such that \[\norm{u}_{L^\infty(\Omega)}\leq C.\]
\end{theorem}

\begin{proof}
Assume that $\Phi:\mathbb{R}\to\mathbb{R}$ is a Lipschitz convex function such that
$\Phi(0) = 0$, then if $u\in H^s_0(\Omega)$, we have, by repeating the proof of Proposition 4 of \cite{BoundedSolnEstimates}, \[\mathcal{E}_a(\Phi(u),v)\leq\mathcal{E}_a(u,v\Phi'(u))\quad\text{ for }v\geq0,v\in H^s_0(\Omega),\text{ weakly in }\Omega.\] 
We can then repeat the proof of Theorem 13 of \cite{BoundedSolnEstimates} using the Moser technique to obtain the theorem.
\end{proof}

Observe that in general our kernel does not satisfy the usual regularity of the kernel of the fractional Laplacian \cite{RosOton2016Survey}\cite{RosOton2014Regularity} or other commonly considered fractional kernels \cite{CaffarelliDeSilvaSavin}\cite{Fall2020CVPDE}, since in general it does not satisfy the ``symmetry'' condition $a(x,y)=a(x,-y)$ unless $a$ is a constant multiple of the kernel of the fractional Laplacian. However, it will still be possible to obtain local  H\"older regularity on the solution with the properties of our kernel, if we assume it is symmetric, i.e. if it satisfies \begin{equation}\label{SymmetryAssump}a(x,y)=a(y,x).\end{equation} Then, we can make use of the symmetric form as given in \eqref{UnweightedGreenEq}, and apply the results of \cite{BoundedSolnEstimates}.

By the Lewy-Stampacchia inequalities, as long as the upper and lower bounds are in $L^p(\Omega)$ for some $p>\frac{d}{2s}$, we can make use of the Dirichlet form nature of the bilinear form, and obtain H\"older regularity on the solutions on balls independently of the boundary conditions and of the regularity of $\partial\Omega$. Then, by Theorem 1.6 of \cite{KassmannDyda}, since the bilinear form satisfies \eqref{kAcond} and \eqref{Markovian1}--\eqref{Markovian2}, in the symmetric case, we have the weak Harnack inequality. Furthermore, as in the classical de Giorgi-Nash-Moser theory, the weak Harnack inequality implies a decay of oscillation-result and local H\"older regularity estimates for weak solutions.

\begin{theorem}[Weak Harnack inequality]\label{WeakHarnack} Let $u$ denote the solutions of the one obstacle problem \eqref{ObsPb}, or the two obstacles problem \eqref{2Obstacles}, or $u=u_i$ for $i=1,\dots,N$ of the N-membranes problem \eqref{NObstacles}, respectively, under the assumptions (a), (b) or (c) above. Suppose the unit ball about the origin $B_1$ is a subset of $\Omega$, and $a$ is symmetric. Then, \[\inf_{B_{1/4}}u\geq c\left(\fint_{B_{1/2}}u(x)^{p_0}\,dx\right)^\frac{1}{p_0}-\sup_{x\in B_{15/16}}\int_{\mathbb{R}^d\backslash B_1}u^-(z)d\mu_x(dz)-\norm{\mathcal{L}_au}_{L^p(B_{15/16})},\] for $d\mu_x(dz)$ a measure depending on $a$ as defined in \cite{KassmannDyda} and \cite{KassmannHarnack}, and the positive constants $p_0\in(0,1)$ and $c$ depend only on $d,s,a_*,a^*$.
\end{theorem}

\begin{theorem}[H\"older regularity]\label{Holder} Let $u$ denote the solutions of the one obstacle problem \eqref{ObsPb}, or the two obstacles problem \eqref{2Obstacles}, or $u=u_i$ for $i=1,\dots,N$ of the N-membranes problem \eqref{NObstacles}, respectively, under the assumptions (a), (b) or (c) above. Suppose $B_\rho\subset\subset\Omega$ is a ball of radius $\rho$ and $a$ is symmetric. Then, there exists $c_\rho\geq0$ and $\beta\in(0,1)$, independent of $u$, such that the following H\"older estimate holds for almost every $x,y\in B_{\rho/2}$: \[|u(x)-u(y)|\leq c_1|x-y|^\beta\left(\norm{u}_{L^\infty(\Omega)}+\norm{\mathcal{L}_au}_{L^p(\Omega)}\right).\]
\end{theorem}

\begin{proof}
Since the Lewy-Stampacchia inequalities in Theorems \ref{LewyStampacchia}, \ref{LewyStampacchia2} and \ref{LewyStampacchiaN} hold a.e. in $B_\rho\subset\Omega$ for $\mathcal{L}_au$ for the one obstacle, the two obstacles problem and the N-membranes problem respectively, and $\mathcal{L}_au=g$ in $B_\rho$, therefore $g$ lies in $L^p(\Omega)$ for $p>\frac{d}{2s}$, and we have the result making use of Theorem \ref{WeakHarnack} following the classical approach.
\end{proof}

\begin{remark}
\begin{enumerate}[label=(\roman*)]
    \item The local H\"older continuity of the solutions of a two membranes problem was obtained for different operators with translation invariant kernels in \cite{CaffarelliDeSilvaSavin}, as well as the local $C^{1,\gamma}$ regularity in the case of the fractional Laplacian as in the case of the regular obstacles of \cite{SilvestrePaper}.
    \item The results in Theorems \ref{WeakHarnack} and \ref{Holder} can be generalized to arbitrary family of measures satisfying \eqref{kAcond} and \eqref{Markovian1}--\eqref{Markovian2}, as given in \cite{KassmannChaker}.
    \item Since the $L^\infty$ bound also works in the non-symmetric case, we conjecture that the weak Harnack inequality and the H\"older continuity are also true without the symmetry assumption, but this is an open problem.
\end{enumerate}
\end{remark}

In the case where $a$ corresponds to the kernel of the fractional Laplacian, $\mathcal{L}_a=(-\Delta)^s$, we can use Corollary 1.15 of \cite{SS1} and apply local elliptic regularity of weak solutions in fractional Sobolev spaces $W^{r,p}$ associated to Dirichlet fractional Laplacian problems as in Theorem 1.4 of \cite{BiccariWarmaZuazua}.

\begin{theorem} Let $u$ denote the solutions of the one obstacle problem \eqref{ObsPb} for $f\in L^{2^\#}(\Omega)$ in the form \[u\in\mathbb{K}^s_\psi(\Omega):\int_{\mathbb{R}^d}D^su\cdot D^s(v-u)\,dx\geq\int_\Omega f(v-u)\quad\forall v\in\mathbb{K}^s_\psi,\] or the corresponding two obstacles problem \eqref{2Obstacles}, or $u=u_i$ for $i=1,\dots,N$ of the corresponding N-membranes problem \eqref{NObstacles}, respectively, under the assumptions (a), (b) or (c) above, with $2^\# \leq p < \infty$ and $0<s<1$. Then, $(-\Delta)^su\in L^p_{loc}(\Omega)$ and $u\in W^{2s,p}_{loc}(\Omega)$. In particular, $u\in C^1(\Omega)$ if $s>1/2$ and $p>d/(2s-1)$, by Theorem 7.57(c) of \cite{Adams}.
\end{theorem}

This theorem, which seems new, is an extension to nonlocal obstacle type problems of the well-known $W^{2,p}_{loc}(\Omega)$ regularity of solutions of the classical local obstacle problem corresponding to $s=1$.

\section{$s$-capacity and Lewy-Stampacchia Inequalities in $H^{-s}(\Omega)$}
In this section, we extend the results on the Lewy-Stampacchia inequalities obtained in the previous section to data in the dual space $H^{-s}(\Omega)$. We first characterize the order dual of $H^s_0(\Omega)$, which is related to the theory of the $s$-capacity. This follows much of the results in the classical obstacle problem \cite{Stampacchia1965}, \cite{AdamsCapacityObstacle} and \cite{ObstacleProblems}. In \cite{AdamsHedberg} and \cite{FukushimaDirichletForms} the more general capacities are considered for general bilinear forms. Recently the fractional capacity for the Neumann problem was considered in \cite{WarmaFracCap}.  In order to extend the results in Theorems \ref{LewyStampacchia}, \ref{LewyStampacchia2} and \ref{LewyStampacchiaN} to data in $H^{-s}(\Omega)$, we may apply the general results of \cite{Mosco1976} for the one obstacle problem and \cite{RodriguesTeymurazyan} for two obstacles.

\subsection{A characterization of the order dual $H^{-s}_\prec(\Omega)$ of $H^s_0(\Omega)$}
Associated with any Dirichlet form, there is a \emph{Choquet capacity}. We denote by $C_s$ the capacity associated to the norm of $H^s_0(\Omega)$. For any compact set $K\subset\Omega$, it is defined by 
\[C_s(K)=\inf\left\{\norm{u}_{H^s_0(\Omega)}^2:u\in H^s_0(\Omega),u\geq1 \text{ a.e. in }K\right\}.\] For an arbitrary open set $G\subset\Omega$, \[C_s(G)=\sup\left\{C_s(K):K\text{ is a compact set in }G\right\}.\]

A function $u\in H^s_0(\Omega)$ is said to be \emph{quasi-continuous} if for every $\varepsilon>0$, there exists an open set $G\subset\Omega$ such that $C_s(G)<\varepsilon$ and $u|_{\Omega\backslash G}$ is continuous. A property is said to hold \emph{quasi-everywhere} (q.e.) if it holds except for a set of capacity zero. 

It is well-known (by \cite{AdamsHedberg} Proposition 6.1.2 page 156 or \cite{FukushimaDirichletForms} Theorem 2.1.3 page 71) that for every $u\in H^s_0(\Omega)$, there exists a unique (up to a set of capacity 0) quasi-continuous function $\bar{u}:\Omega\to\mathbb{R}$ such that $\bar{u}=u$ a.e. on $\Omega$. 
Therefore, we have the following theorem (see also Theorem 3.7 of \cite{WarmaFracCap}).

\begin{theorem}\label{QCApprox} For every function $u\in H^s_0(\Omega)$, there exists a unique (up to q.e. equivalence) $\bar{u}:\Omega\to\mathbb{R}$ quasi-continuous function such that $u=\bar{u}$ a.e. in $\Omega$.
\end{theorem}

Thus, it makes sense to identify a function $u\in H^s_0(\Omega)$ with the class of quasi-continuous functions that are equivalent quasi-everywhere. Denote the space of such equivalent classes by $Q_s(\Omega)$. Then, for every element $u\in H^s_0(\Omega)$, there is an associated $\bar{u}\in Q_s(\Omega)$.

Define the space $L^2_{C_s}(\Omega)$ by 
\[L^2_{C_s}(\Omega)=\{\phi\in Q_s(\Omega):\exists u\in H^s_0(\Omega):\bar{u}\geq|\phi|\text{ q.e. in }\Omega\}\] and \[R_{C_s}(\phi)=\inf\{\norm{u}_{H^s_0(\Omega)}:u\in H^s_0(\Omega),\bar{u}\geq|\phi|\text{ q.e.}\},\] which is a norm that makes $L^2_{C_s}(\Omega)$ a Banach space (see Proposition 1.2 of \cite{AttouchPicard}). We want to show that the dual space of $L^2_{C_s}(\Omega)$ can be identified with the order dual of $H^s_0(\Omega)$, i.e. \[[L^2_{C_s}(\Omega)]'=H^{-s}(\Omega)\cap M(\Omega)=H^{-s}_\prec(\Omega)=[H^{-s}(\Omega)]^+-[H^{-s}(\Omega)]^+,\] where $M(\Omega)$ is the set of bounded measures in $\Omega$. Then we have the following result by Theorem \ref{QCApprox} (corresponding to the classical case in \cite{AttouchPicard} Proposition 1.7).
\begin{proposition}\label{QCDensity}The injection of $H^s_0(\Omega)\cap C_c(\Omega)\hookrightarrow L^2_{C_s}(\Omega)$ is dense.\end{proposition}
\begin{proof} This simply follows from Theorem \ref{QCApprox}, since $u\geq0$ a.e. on $\Omega$ implies $u\geq0$ q.e. on $\Omega$.\end{proof}

For $K\subset\Omega$, recall that one says that $u\succeq0$ on $K$ (or $u\geq0$ on $K$ in the sense of $H^s_0(\Omega)$) if there exists a sequence of Lipschitz functions $u_k\to u$ in $H^s_0(\Omega)$ such that $u_k\geq0$ on $K$.

Let $K\subset\Omega$ be any compact subset. Define the nonempty closed convex set of $H^s_0(\Omega)$ by \[\mathbb{K}^s_K=\{v\in H^s_0(\Omega):v\succeq1\text{ on }K\}.\] Consider the following variational inequality \begin{equation}\label{CapVarIneq}u\in\mathbb{K}^s_K:\mathcal{E}_a(u,v-u)\geq0,\quad\forall v\in\mathbb{K}^s_K.\end{equation} This variational inequality has clearly a unique solution and consequently we can also extend to the $s$-fractional framework the following theorem which is due to Stampacchia \cite{Stampacchia1965} in the case $s=1$.

\begin{theorem}[Radon measure for the bilinear form $\mathcal{E}_a$]\label{RadonMeasure}
For any compact $K\subset\Omega$, the unique solution $u$ of \eqref{CapVarIneq}, which is called the \emph{$(s,a)$-capacitary potential} of $K$, is such that \[u=1\text{ on }K\text{ (in the sense of }H^s_0(\Omega))\]\[\mu=\mathcal{L}_au\geq0\text{ with }supp(\mu)\subset K.\] Moreover, for the non-negative Radon measure $\mu$, one has \[C_s^a(K)=\mathcal{E}_a(u,u)=\int_\Omega d\mu=\mu(K)\] and this number is called the \emph{$(s,a)$-capacity} of $K$ with respect to $\mathcal{E}_a(\cdot,\cdot)$ (or to the operator $\mathcal{L}_a$).
\end{theorem}

\begin{proof}
The proof follows a similar approach to the classical case (\cite{Stampacchia1965} Theorem 3.9 or \cite{ObstacleProblems} Theorem 8.1).
Taking $v=u\wedge1=u-(u-1)^+\in\mathbb{K}^s_K$ in \eqref{CapVarIneq}, one has \[a_*\norm{(u-1)^+}^2_{H^s_0(\Omega)}\leq\mathcal{E}_a(u-1,(u-1)^+)=\mathcal{E}_a(u,(u-1)^+)\leq0\] since the $s$-grad of a constant is zero. Hence $u\preceq1$ in $\Omega$. But $u\in\mathbb{K}^s_K$, so $u\succeq1$ on $K$. Therefore, the first result $u=1$ on $K$ follows.

For the second result, set $v=u+\varphi\in\mathbb{K}^s_K$ in \eqref{CapVarIneq} with an arbitrary $\varphi\in \mathcal{D}(\Omega)$, $\varphi\geq0$. Then, by the Riesz-Schwartz theorem (see for instance \cite{AdamsHedberg} Theorem 1.1.3), there exists a non-negative Radon measure $\mu$ on $\Omega$ such that \[\langle\mathcal{L}_au,\varphi\rangle=\mathcal{E}_a(u,\varphi)=\int_\Omega\varphi \,d\mu,\quad\forall\varphi\in\mathcal{D}(\Omega).\]
Moreover, for $x\in\Omega\backslash K$, there is a neighbourhood $O\subset\Omega\backslash K$ of $x$ so that $u+\varphi\in\mathbb{K}^s_K$ for any $\varphi\in\mathcal{D}(O)$. Therefore,\[\mathcal{E}_a(u,\varphi)=0,\quad\forall\varphi\in\mathcal{D}(\Omega\backslash K)\]which means $\mu=\mathcal{L}_au=0$ in $\Omega\backslash K$. Therefore, $supp(\mu)\subset K$ and the third result follows immediately.
\end{proof}

We observe that when $a$ corresponds to the kernel of the fractional Laplacian, the $(s,a)$-capacity corresponds to the $s$-capacity and the \emph{$s$-capacitary potential} of a compact set $K$
is the solution of the obstacle problem \eqref{CapVarIneq} when the bilinear form is the inner product in $H^s_0(\Omega)$ and we have a simple comparison of the capacities in the following proposition.
\begin{proposition}\label{CapProp5.4}
For any compact subset $E\subset\Omega$, \[a_*C_s(E)\leq C_s^a(E)\leq \frac{{a^*}^2}{a_*}C_s(E).\]
\end{proposition}
\begin{proof}
Let $u$ be the $(s,a)$-capacitary potential of $E$, and $\bar{u}$ be the $s$-capacitary potential of $E$. Since $\bar{u}\succeq1$ on $E$, we can choose $v=\bar{u}\in\mathbb{K}^s_{E}$ in \eqref{CapVarIneq} to get
\begin{align*}C_s^a(E)=\mathcal{E}_a(u,u)\leq&\mathcal{E}_a(u,\bar{u})\\\leq& a^*\left(\int_{\mathbb{R}^d}|D^su|^2\right)^\frac{1}{2}\left(\int_{\mathbb{R}^d}|D^s\bar{u}|^2\right)^\frac{1}{2}\\\leq& \frac{a_*}{2}\int_{\mathbb{R}^d}|D^su|^2+\frac{{a^*}^2}{2a_*}\int_{\mathbb{R}^d}|D^s\bar{u}|^2\\\leq&\frac{1}{2} \mathcal{E}_a(u,u)+\frac{{a^*}^2}{2a_*}C_s(E)\\=&\frac{1}{2}C_s^a(E)+\frac{{a^*}^2}{2a_*}C_s(E)\end{align*} by Cauchy-Schwarz inequality and the coercivity of $a$. Similarly, we can choose $v=u\in\mathbb{K}^s_{E}$ for \eqref{CapVarIneq} with $a$ being the identity for $C_s(E)$ to get, using the coercivity of $a$,
\begin{align*}C_s(E)=\mathcal{E}_a(\bar{u},\bar{u})\leq&\mathcal{E}_a(\bar{u},u)\\\leq&\left(\int_{\mathbb{R}^d}|D^s\bar{u}|^2\right)^\frac{1}{2}\left(\int_{\mathbb{R}^d}|D^su|^2\right)^\frac{1}{2}\\\leq& \frac{1}{2}\int_{\mathbb{R}^d} |D^s\bar{u}|^2+\frac{1}{2}\int_{\mathbb{R}^d}|D^su|^2\\\leq& \frac{1}{2}C_s(E)+\frac{1}{2a_*}\mathcal{E}_a(u,u)\\=& \frac{1}{2}C_s(E)+\frac{1}{2a_*}C_s^a(E).\end{align*}
\end{proof}

Using this definition of the Radon measure, we recall that two quasi-continuous functions which are equal (or, $\leq$) $\mu$-a.e. on an open subset of $\mathbb{R}^d$ are also equal (or, $\leq$) q.e. on that set (see \cite{FukushimaDirichletForms} Lemma 2.1.4).

Recall that a Radon measure $\mu$ is said to be of finite energy relatively to $H^s_0(\Omega)$ if its restriction to $H^s_0(\Omega)\cap C_c(\Omega)$ is continuous for the topology of $H^s_0(\Omega)$, by means of \[\langle\mu,v\rangle=\int_\Omega v\,d\mu,\quad\forall v\in H^s_0(\Omega)\cap C_c(\Omega).\] Such a finite energy measure can in fact be defined for any Dirichlet form $\mathcal{E}$ (see \cite{FukushimaDirichletForms} Section 2.2 and Example 2.2.1 pages 87--91). We denote by $E^+(H^s_0(\Omega))$ the cone of positive finite energy measures relative to $H^s_0(\Omega)$. Then $\mu$ is of finite energy if and only if there exists $w_\mu \in H^{-s}(\Omega)$ such that \[\langle w_\mu,v\rangle=\int_\Omega v\,d\mu\quad\forall v\in H^s_0(\Omega)\cap C_c(\Omega),\] and $E^+(H^s_0(\Omega))$ can be identified with $[H^{-s}(\Omega)]^+$, the positive cone of $H^{-s}(\Omega)=[H^s_0(\Omega)]'$, by the mapping $\mu\mapsto w_\mu$. Moreover, whenever $\mu\in E^+(H^s_0(\Omega))$, the mapping $u\in H^s_0(\Omega)\mapsto\bar{u}$ is continuous from $H^s_0(\Omega)$ in $L^1(\mu)$ and whenever $u\in H^s_0(\Omega)$, $\int_\Omega\bar{u}\,d\mu=\langle w_\mu,v\rangle$. Note that in the particular case of the space $H^s_0(\Omega)$, the mapping $u\in H^s_0(\Omega)\mapsto\bar{u}\in L^1(\mu)$ is compact; this follows from the fact that $\int_\Omega|\bar{u}_n|\,d\mu=\langle w_\mu,|u_n|\rangle$ and that if $u_n\rightharpoonup0$ in $H^s_0(\Omega)$ then $|u_n|\rightharpoonup0$ in $H^s_0(\Omega)$.

Extending these results to $L^2_{C_s}(\Omega)$, we have the following result.
\begin{proposition}\label{AttouchPicardProp12} Let $\mu\in E^+(H^s_0(\Omega))$. Then $L^2_{C_s}(\Omega)\subset L^1(\mu)$ and this inclusion is continuous.\end{proposition}

\begin{proof} Let $u\in L^2_{C_s}(\Omega)$. There exists $v\in H^s_0(\Omega)$ such that $\bar{v}\geq|u|$ a.e., and therefore $\mu$-q.e.. Since $\bar{v}\in L^1(\mu)$, $u\in L^1(\mu)$.

Let $(u_n)$ be a sequence in $L^2_{C_s}(\Omega)$ such that $R_{C_s}(u_n)\to0$. Then there exists $(v_n)\in H^s_0(\Omega)$ such that $\bar{v}_n\geq|u_n|$ q.e., and therefore $\mu$-q.e., and $\norm{v_n}_{H^s_0(\Omega)}\to0$. As a result, $\int_\Omega|u_n|\,d\mu\leq\int_\Omega\bar{v}_n\,d\mu=\langle w_\mu,v_n\rangle\leq\|w_\mu\|_{H^{-s}(\Omega)}\norm{v_n}_{H^s_0(\Omega)}\to0$. Therefore $u_n\to0$ in $L^1(\mu)$.\end{proof}

Having these results, we can now identify the dual space of $L^2_{C_s}(\Omega)$ with the order dual of $H^s_0(\Omega)$, as given in the following theorem.

\begin{theorem}[Characterization of Order Dual]\label{OrderDualChar}The dual of $L^2_{C_s}(\Omega)$ is the space of finite energy measures $E^+(H^s_0(\Omega))-E^+(H^s_0(\Omega))$, that is identified with the order dual $H^{-s}_\prec(\Omega)$ of $H^s_0(\Omega)$. More precisely, $L\in [L^2_{C_s}(\Omega)]'$ if and only if there is a unique $\mu$ such as $|\mu|\in E^+(H^s_0(\Omega))$ and $L(\phi)=\int_\Omega\phi \,d\mu$ for all $\phi\in L^2_{C_s}(\Omega)$. In addition, the norm of $L$ in $[L^2_{C_s}(\Omega)]'$ is such that $\norm{L}=\|w_{|\mu|}\|_{H^{-s}(\Omega)}$.\end{theorem}

\begin{proof} According to Proposition \ref{QCDensity}, $C_c(\Omega)$ is dense in $L^2_{C_s}(\Omega)$ and moreover this injection is continuous; therefore the dual of $L^2_{C_s}(\Omega)$ is a space of measures.

Let $\mu$ be a Radon measure such that $|\mu|\in E^+(H^s_0(\Omega))$. For any $\phi\in L^2_{C_s}(\Omega)$ ($\phi$ is then $\mu$ integrable by Proposition \ref{AttouchPicardProp12}), set $L(\phi)=\int_\Omega\phi \,d\mu$. For any $v\in H^s_0$ such that $\bar{v}\geq|\phi|$ quasi-everywhere, so $\mu$-a.e., we have \[|L(\phi)|=\left|\int_\Omega\phi \,d\mu\right|\leq\int_\Omega|\phi|d|\mu|\leq\bar{v}d|\mu|=(w_{|\mu|},v)\leq\|w_{|\mu|}\|_{H^{-s}(\Omega)}\norm{v_n}_{H^s_0(\Omega)},\] so $|L(\phi)|\leq\|w_{|\mu|}\|R_{C_s}(\phi)$. Therefore $L\in [L^2_{C_s}(\Omega)]'$ and $\norm{L}\leq\|w_{|\mu|}\|_{H^{-s}(\Omega)}$.

Conversely, suppose $L\in [L^2_{C_s}(\Omega)]'$. Let $K$ be a compact subset of $\Omega$ and $\psi\in\mathcal{D}(\Omega)$ such that $\psi\geq1$ in $K$. Whenever $\phi\in C_s(K)$ such that $\norm{\phi}_\infty\leq1$, we have $\psi\geq|\phi|$ in $\Omega$, therefore \[|L(\phi)|\leq\norm{L}\cdot R_{C_s}(\psi).\] We deduce that there exists a Radon measure $\mu$ in $\Omega$ such that \[\forall\phi\in C_c(\Omega),\quad L(\phi)=\int_\Omega\phi \,d\mu.\] In addition $|\mu|\in E^+(H^s_0)$ because, whenever $u\in H^s_0(\Omega)\cap C_c(\Omega)$, \begin{align*}\left|\int_\Omega ud|\mu|\right|&\leq\int_\Omega|u|d|\mu|\\&=\sup\{L(\phi):\phi\in C_c(\Omega),|\phi|\leq|u|\}\\&\leq\sup\{\norm{L}\cdot R_{C_s}(\phi):|\phi|\leq|u|\}\\&\leq\norm{L}\cdot R_{C_s}(u)\\&\leq\norm{L}\cdot \norm{u}_{H^s_0(\Omega)}\end{align*} and also $\|w_{|\mu|}\|_{H^{-s}(\Omega)}\leq\norm{L}$.

Finally, it follows, from the density of $C_c(\Omega)$ in $L^2_{C_s}(\Omega)$ and Proposition \ref{AttouchPicardProp12}, that \[\forall\phi\in  L^2_{C_s}(\Omega),\quad L(\phi)=\int_\Omega\phi \,d\mu.\]
\end{proof}

\subsection{Lewy-Stampacchia inequalities in $H^{-s}_\prec(\Omega)$}

For completeness, we state the Lewy-Stampacchia inequalities in the dual space $H^{-s}_\prec(\Omega)$.

\begin{theorem}\label{ObsPbH-s} The unique solution $u$ to the obstacle problem \eqref{ObsPb} with compatible obstacle $\psi\in H^s(\mathbb{R}^d)$ and $F,\mathcal{L}_a\psi\in H^{-s}_\prec(\Omega)$, satisfies \begin{equation}\label{ObsPbH-sEq}F\leq \mathcal{L}_au\leq F\vee\mathcal{L}_a\psi\quad\text{ in }H^{-s}_\prec(\Omega).\end{equation}
\end{theorem}

\begin{proof}
Since $\mathcal{L}_a$ is strictly T-monotone, this is a direct consequence of the abstract Lewy-Stampacchia inequality obtained by Mosco in \cite{Mosco1976} (see also Theorem 2.1 of \cite{ObstacleProblems}). 
\end{proof}

We next consider the generalizations to the two obstacles problem and to the N-membranes problem.

Similarly, as a direct consequence of Theorem 4.2 of \cite{RodriguesTeymurazyan}, we may also state the Lewy-Stampacchia inequality for the two obstacles problem.

\begin{theorem}\label{2ObstaclesH-s}
The solution $u$ of the two obstacles problem \[u\in \mathbb{K}^s_{\varphi,\psi}:\quad\mathcal{E}_a(u,v-u)\geq\langle F,v-u\rangle\quad\forall v\in\mathbb{K}^s_{\varphi,\psi}\] with $\mathbb{K}^s_{\varphi,\psi}$ given by \eqref{2ObstaclesConvex} with data $F\in H^{-s}_\prec(\Omega)$ and $\mathcal{L}_a\psi,\mathcal{L}_a\varphi\in H^{-s}_\prec(\Omega)$ satisfies \begin{equation}\label{2ObstaclesH-sEq}F\wedge\mathcal{L}_a\varphi\leq \mathcal{L}_au\leq F\vee\mathcal{L}_a\psi\quad\text{ in }H^{-s}_\prec(\Omega).\end{equation} 
\end{theorem}

Then, applying the general Lewy-Stampacchia inequalities for the one obstacle and for the two obstacles problem iteratively in the previous theorem as in the proof of Theorem \ref{LewyStampacchiaN}, we obtain 
\begin{theorem}\label{NmembranesH-s}
The solution $u$ of the N-membranes problem \[u\in \mathbb{K}^s_N:\quad\sum_{i=1}^N\mathcal{E}_a(u_i,v_i-u_i)\geq\sum_{i=1}^N\langle F^i,v_i-u_i\rangle\quad\forall (v_1,\dots v_N)\in\mathbb{K}^s_N\] with $\mathbb{K}^s_N$ given by \eqref{NObstaclesConvex} with data $F=(F^1,\dots,F^N)$ for $F^i\in H^{-s}_\prec(\Omega)$ satisfies \begin{align*}F^1\wedge \mathcal{L}_au_1&\leq F^1\vee\dots\vee F^N\\F^1\wedge F^2\leq \mathcal{L}_au_2&\leq F^2\vee\dots\vee F^N\\\vdots\\F^1\wedge\dots\wedge F^{N-1}\leq \mathcal{L}_au_{N-1}&\leq F^{N-1}\vee F^N\\F^1\wedge\dots\wedge F^N\leq \mathcal{L}_au_N&\leq F^N\end{align*} in $H^{-s}_\prec(\Omega)$.
\end{theorem}

\begin{remark}
In the symmetric case, the Lewy-Stampacchia inequalities also follow from the general results of \cite{GMosconi}. The application to Theorem \ref{NmembranesH-s} for the N-membranes problem is new.
\end{remark}

\subsection{The $\mathcal{E}_a$ obstacle problem and the $s$-capacity}

As a simple application of $s$-capacity, we consider the corresponding nonlocal obstacle problem extending some results of \cite{Stampacchia1965} and \cite{AdamsCapacityObstacle} (see also \cite{ObstacleProblems}). In this section we start by the following comparison property for the $(s,a)$-capacity which proof is exactly as in Theorem 3.10 of \cite{Stampacchia1965}, which states that in the case when the kernel $a$ is symmetric the $(s,a)$-capacity is an increasing set function.

\begin{proposition}
For any compact subsets $E_1\subset E_2\subset\Omega$, \[C_s^a(E_1)\leq \left(1+\frac{M}{a_*}\right)^2 C_s^a(E_2),\] where $M=\sup\frac{1}{2}(\mathcal{E}_a(u,v)-\mathcal{E}_a(v,u))$ for $u,v$ such that $\norm{u}_{H^s_0(\Omega)}=\norm{v}_{H^s_0(\Omega)}=1$.
\end{proposition}

\begin{theorem}
Let $\psi$ be an arbitrary function in $L^2_{C_s}(\Omega)$. Suppose that the closed convex set $\bar{\mathbb{K}}_\psi^s$ is such that \[\bar{\mathbb{K}}_\psi^s=\{v\in H^s_0(\Omega):\bar{v}\geq\psi \text{ q.e. in }\Omega\}\neq\emptyset.\] Then there is a unique solution to \begin{equation}\label{ComparisonCap1}u\in\bar{\mathbb{K}}_\psi^s:\quad\mathcal{E}_a(u,v-u)\geq0,\quad\forall v\in\bar{\mathbb{K}}_\psi^s,\end{equation} which is non-negative and such that \begin{equation}\label{ComparisonCap2}\norm{u}_{H^s_0(\Omega)}\leq(a^*/a_*)\norm{\psi^+}_{L^2_{C_s}(\Omega)};\end{equation} there is a unique measure $\mu=\mathcal{L}_au\geq0$, concentrated on the coincidence set $\{u=\psi\}=\{u=\psi^+\}$, verifying \begin{equation}\label{ComparisonCap3}\mathcal{E}_a(u,v)=\int_\Omega\bar{v}\,d\mu,\quad\forall v\in H^s_0(\Omega),\end{equation} and \begin{equation}\label{ComparisonCap4}\mu(K)\leq\left(\frac{a^{*2}}{a_*^{3/2}}\right)\norm{\psi^+}_{L^2_{C_s}(\Omega)}\left[C_s^a(K)\right]^{1/2},\quad\forall K\subset\subset\Omega,\end{equation} in particular $\mu$ does not charge on sets of capacity zero.
\end{theorem}

\begin{proof}
By the maximum principle Proposition \ref{ObsPbProperties}(ii), taking $v=u+u^-$, the solution is non-negative. Hence, the variational inequality \eqref{ComparisonCap1} is equivalent to solving the variational inequality with $\bar{\mathbb{K}}_\psi^s$ replaced by $\bar{\mathbb{K}}_{\psi^+}^s$. Since $\psi^+\in L^2_{C_s}(\Omega)$, by definition of $L^2_{C_s}(\Omega)$, $\bar{\mathbb{K}}_{\psi^+}^s\neq\emptyset$ and we can apply the Stampacchia theorem to obtain a unique non-negative solution.

Since $\mathcal{E}_a(u,v-u)\geq0$, \[a_*\norm{u}_{H^s_0(\Omega)}^2\leq\mathcal{E}_a(u,u)\leq\mathcal{E}_a(u,v)\leq a^*\norm{u}_{H^s_0(\Omega)}\norm{v}_{H^s_0(\Omega)},\] we have \[\norm{u}_{H^s_0(\Omega)}\leq(a^*/a_*)\norm{v}_{H^s_0(\Omega)},\quad\forall v\in \bar{\mathbb{K}}_{\psi^+}^s,\] which, by the definition of the $L^2_{C_s}(\Omega)$ norm of $\psi^+$, gives \eqref{ComparisonCap2}.

The existence of Radon measure for \eqref{ComparisonCap3} follows exactly as in Theorem \ref{RadonMeasure}. Finally, recalling the definitions, it is sufficient to prove \eqref{ComparisonCap4} for any compact subset $K\subset\Omega$. But this follows from \[\mu(K)\leq\int_\Omega\bar{v}\,d\mu=\mathcal{E}_a(u,v)\leq a^*\norm{u}_{H^s_0(\Omega)}\norm{v}_{H^s_0(\Omega)}\leq(a^{*2}/a_*)\norm{\psi^+}_{L^2_{C_s}(\Omega)}\norm{v}_{H^s_0(\Omega)},\quad\forall v\in \mathbb{K}_K^s.\] Now, recall from Proposition \ref{CapProp5.4} that we have \[C_s^a(K)\geq a_*C_s(K)=a_*\inf_{v\in\mathbb{K}_K^s}\norm{v}_{H^s_0(\Omega)}^2\] thereby obtaining \eqref{ComparisonCap4}.
\end{proof}

\begin{corollary}
If $u$ and $\hat{u}$ are the solutions to \eqref{ComparisonCap1} with non-negative compatible obstacles $\psi$ and $\hat{\psi}$ in $L^2_{C_s}(\Omega)$ respectively, then \[\norm{u-\hat{u}}_{H^s_0(\Omega)}\leq k\|\psi-\hat{\psi}\|_{L^2_{C_s}(\Omega)}^{1/2},\] where \[k=(a^*/a_*)\left[\norm{\psi}_{L^2_{C_s}(\Omega)}+\|\hat{\psi}\|_{L^2_{C_s}(\Omega)}\right]^{1/2}.\]
\end{corollary}

\begin{proof}
Since $supp(\mu)\subset\{u=\psi\}$ and $supp(\hat{\mu})\subset\{\hat{u}=\hat{\psi}\}$ (where $\mu=\mathcal{L}_au$ and $\hat{\mu}=\mathcal{L}_a\hat{u}$), for an arbitrary $v\in\bar{K}^s_{|\psi-\hat{\psi}|}$, we have \begin{align*}a_*\norm{u-\hat{u}}_{H^s_0(\Omega)}^2&\leq\mathcal{E}_a(u-\hat{u},u-\hat{u})\\&=\mathcal{E}_a(u,u-\hat{u})-\mathcal{E}_a(\hat{u},u-\hat{u})\\&=\int_\Omega(u-\hat{u})\,d\mu-\int_\Omega(u-\hat{u})d\hat{\mu}\text{ by setting }\bar{v}=u-\hat{u} \text{ for }\mu,\hat{\mu}\text{ in \eqref{ComparisonCap3}}\\&\leq\int_\Omega(\psi-\hat{\psi})\,d\mu-\int_\Omega(\psi-\hat{\psi})d\hat{\mu}\\&\leq\int_\Omega|\psi-\hat{\psi}|d(\mu+\hat{\mu})\\&\leq\int_\Omega\bar{v}d(\mu+\hat{\mu}) \text{ since }v\in\bar{K}^s_{|\psi-\hat{\psi}|}\\&=\int_\Omega\bar{v}\,d\mu+\int_\Omega\bar{v}d\hat{\mu}\\&=\mathcal{E}_a(u,v)+\mathcal{E}_a(\hat{u},v)\\&\leq a^*\left[\norm{u}_{H^s_0(\Omega)}+\norm{\hat{u}}_{H^s_0(\Omega)}\right]\norm{v}_{H^s_0(\Omega)}\\&\leq\frac{a^{*2}}{a_*}\left[\norm{\psi}_{L^2_{C_s}(\Omega)}+\|\hat{\psi}\|_{L^2_{C_s}(\Omega)}\right]\norm{v}_{H^s_0(\Omega)}\text{ by \eqref{ComparisonCap2}}\\&\leq\frac{a^{*2}}{a_*}\left[\norm{\psi}_{L^2_{C_s}(\Omega)}+\|\hat{\psi}\|_{L^2_{C_s}(\Omega)}\right]\|\psi-\hat{\psi}\|_{L^2_{C_s}(\Omega)}
\end{align*}
since $v$ is arbitrary in $\bar{K}^s_{|\psi-\hat{\psi}|}$, by the definition of the norm of $|\psi-\hat{\psi}|$ in $L^2_{C_s}(\Omega)$.
\end{proof}

\begin{remark}
Further properties on the $s$-capacity and the regularity of the solution to the $s$-obstacle problem are an interesting topic to be developed. For instance, as in the classical local case $s=1$ \cite{AdamsCapacityObstacle}, it would be interesting to show that $\psi$ is compatible, i.e. $\bar{\mathbb{K}}^s_\psi\neq\emptyset$, if and only if \[\int_0^\infty C_s(\{|\psi^+|>t\})dt^2<\infty.\] \end{remark}

\noindent \textbf{Acknowledgements.} 
C. Lo acknowledges the FCT PhD fellowship in the framework of the LisMath doctoral programme at the University of Lisbon. The research of J. F. Rodrigues was partially done under the framework of the Project PTDC/MATPUR/28686/2017 at CMAFcIO/ULisboa. Also, many thanks to Alan Aw, Pedro Campos, Lisa Santos, Shuang Wu, and an anonymous referee, for providing useful comments.

\bibliographystyle{plain}


\begin{thebibliography}{10}

\bibitem{AbatangeloRosOton}
Nicola Abatangelo and Xavier Ros-Oton.
\newblock Obstacle problems for integro-differential operators: higher
  regularity of free boundaries.
\newblock {\em Adv. Math.}, 360:106931, 61, 2020.

\bibitem{AdamsCapacityObstacle}
David~R. Adams.
\newblock Capacity and the obstacle problem.
\newblock {\em Appl. Math. Optim.}, 8(1):39--57, 1982.

\bibitem{AdamsHedberg}
David~R. Adams and Lars~Inge Hedberg.
\newblock {\em Function spaces and potential theory}, volume 314 of {\em
  Grundlehren der Mathematischen Wissenschaften [Fundamental Principles of
  Mathematical Sciences]}.
\newblock Springer-Verlag, Berlin, 1996.

\bibitem{Adams}
Robert~A. Adams.
\newblock {\em Sobolev spaces}.
\newblock Academic Press [A subsidiary of Harcourt Brace Jovanovich,
  Publishers], New York-London, 1975.
\newblock Pure and Applied Mathematics, Vol. 65.

\bibitem{AntilRautenberg}
Harbir Antil and Carlos~N. Rautenberg.
\newblock Fractional elliptic quasi-variational inequalities: theory and
  numerics.
\newblock {\em Interfaces Free Bound.}, 20(1):1--24, 2018.

\bibitem{AttouchPicard}
H\'{e}dy Attouch and Colette Picard.
\newblock In\'{e}quations variationnelles avec obstacles et espaces
  fonctionnels en th\'{e}orie du potentiel.
\newblock {\em Applicable Anal.}, 12(4):287--306, 1981.

\bibitem{Rodrigues2005NMembrane}
Assis Azevedo, Jos\'{e}-Francisco Rodrigues, and Lisa Santos.
\newblock The {$N$}-membranes problem for quasilinear degenerate systems.
\newblock {\em Interfaces Free Bound.}, 7(3):319--337, 2005.

\bibitem{BellidoCuetoMoraCorral2021CVPDE}
Jos\'{e}~C. Bellido, Javier Cueto, and Carlos Mora-Corral.
\newblock {$\Gamma $}-convergence of polyconvex functionals involving
  s-fractional gradients to their local counterparts.
\newblock {\em Calc. Var. Partial Differential Equations}, 60(1):Paper No. 7,
  2021.

\bibitem{BiccariWarmaZuazua}
Umberto Biccari, Mahamadi Warma, and Enrique Zuazua.
\newblock Local elliptic regularity for the {D}irichlet fractional {L}aplacian.
\newblock {\em Adv. Nonlinear Stud.}, 17(2):387--409, 2017.

\bibitem{FiniteElement}
Andrea Bonito, Wenyu Lei, and Abner~J. Salgado.
\newblock Finite element approximation of an obstacle problem for a class of
  integro-differential operators.
\newblock {\em ESAIM Math. Model. Numer. Anal.}, 54(1):229--253, 2020.



\bibitem{FracLapMaxPrinciple}
Xavier Cabr\'{e} and Yannick Sire.
\newblock Nonlinear equations for fractional {L}aplacians, {I}: {R}egularity,
  maximum principles, and {H}amiltonian estimates.
\newblock {\em Ann. Inst. H. Poincar\'{e} Anal. Non Lin\'{e}aire},
  31(1):23--53, 2014.

\bibitem{CaffarelliDeSilvaSavin}
Luis Caffarelli, Daniela De~Silva, and Ovidiu~V. Savin.
\newblock The two membranes problem for different operators.
\newblock {\em Ann. Inst. H. Poincar\'{e} Anal. Non Lin\'{e}aire},
  34(4):899--932, 2017.

\bibitem{CaffarelliSalsaSilvestre}
Luis~A. Caffarelli, Sandro Salsa, and Luis Silvestre.
\newblock Regularity estimates for the solution and the free boundary of the
  obstacle problem for the fractional {L}aplacian.
\newblock {\em Invent. Math.}, 171(2):425--461, 2008.

\bibitem{KassmannChaker}
Jamil Chaker and Moritz Kassmann.
\newblock Nonlocal operators with singular anisotropic kernels.
\newblock {\em Comm. Partial Differential Equations}, 45(1):1--31, 2020.

\bibitem{Comi2}
Giovanni~E. Comi and Giorgio Stefani.
\newblock A distributional approach to fractional {S}obolev spaces and fractional
  variation: asymptotics i.
\newblock {\em arXiv: 1910.13419 [math.FA]}, 2019.

\bibitem{Comi1}
Giovanni~E. Comi and Giorgio Stefani.
\newblock A distributional approach to fractional {S}obolev spaces and fractional
  variation: Existence of blow-up.
\newblock {\em Journal of Functional Analysis}, 277(10):3373 -- 3435, 2019.

\bibitem{DanielliPetrosyanPopObsPbMarkov}
Donatella Danielli, Arshak Petrosyan, and Camelia~A. Pop.
\newblock Obstacle problems for nonlocal operators.
\newblock In {\em New developments in the analysis of nonlocal operators},
  volume 723 of {\em Contemp. Math.}, pages 191--214. Amer. Math. Soc.,
  Providence, RI, 2019.

\bibitem{DanielliSalsa}
Donatella Danielli and Sandro Salsa.
\newblock Obstacle problems involving the fractional {L}aplacian.
\newblock In {\em Recent developments in nonlocal theory}, pages 81--164. De
  Gruyter, Berlin, 2018.


\bibitem{DElia2020unified}
Marta D'Elia, Mamikon Gulian, Hayley Olson, and George~Em Karniadakis.
\newblock A unified theory of fractional, nonlocal, and weighted nonlocal
  vector calculus.
\newblock {\em arXiv preprint arXiv:2005.07686}, 5 2020.

\bibitem{delia2021analysis}
Marta D'Elia and Mamikon Gulian.
\newblock Analysis of anisotropic nonlocal diffusion models: Well-posedness of
  fractional problems for anomalous transport.
\newblock {\em arXiv preprint arXiv:2101.04289}, 1 2021.

\bibitem{Demengel}
Fran\c{c}oise Demengel and Gilbert Demengel.
\newblock {\em Functional spaces for the theory of elliptic partial
  differential equations}.
\newblock Universitext. Springer, London; EDP Sciences, Les Ulis, 2012.
\newblock Translated from the 2007 French original by Reinie Ern\'{e}.

\bibitem{KumagaiDeuschelNonsymmetricKernel}
Jean-Dominique Deuschel and Takashi Kumagai.
\newblock Markov chain approximations to nonsymmetric diffusions with bounded
  coefficients.
\newblock {\em Comm. Pure Appl. Math.}, 66(6):821--866, 2013.

\bibitem{HitchhikerGuide}
{Eleonora Di~Nezza, Giampiero Palatucci, and Enrico Valdinoci}
\newblock Hitchhiker's guide to the fractional {S}obolev spaces.
\newblock {\em Bull. Sci. Math. }136(5):521--573, 2012.

\bibitem{DuGunzburgerLehoucqZhouNonlocalDiffusion}
Qiang Du, Max Gunzburger, Richard~B. Lehoucq, and Kun Zhou.
\newblock Analysis and approximation of nonlocal diffusion problems with volume
  constraints.
\newblock {\em SIAM Rev.}, 54(4):667--696, 2012.

\bibitem{DuGunzburgerLehoucqZhouNonlocalVectorCalculus}
Qiang Du, Max Gunzburger, Richard~B. Lehoucq, and Kun Zhou.
\newblock A nonlocal vector calculus, nonlocal volume-constrained problems, and
  nonlocal balance laws.
\newblock {\em Math. Models Methods Appl. Sci.}, 23(3):493--540, 2013.

\bibitem{KassmannDyda}
Bart{\l}omiej Dyda and Moritz Kassmann.
\newblock Regularity estimates for elliptic nonlocal operators.
\newblock {\em Anal. PDE}, 13(2):317--370, 2020.

\bibitem{Fall2020CVPDE}
Mouhamed~Moustapha Fall.
\newblock Regularity results for nonlocal equations and applications.
\newblock {\em Calc. Var. Partial Differential Equations}, 59(5):Paper No. 181,
  53, 2020.

\bibitem{FelsingerKassmannVoigt}
Matthieu Felsinger, Moritz Kassmann, and Paul Voigt.
\newblock The {D}irichlet problem for nonlocal operators.
\newblock {\em Math. Z.}, 279(3-4):779--809, 2015.

\bibitem{FukushimaDirichletForms}
Masatoshi Fukushima, Yoichi Oshima, and Masayoshi Takeda.
\newblock {\em Dirichlet forms and symmetric {M}arkov processes}, volume~19 of
  {\em De Gruyter Studies in Mathematics}.
\newblock Walter de Gruyter \& Co., Berlin, extended edition, 2011.

\bibitem{GMosconi}
Nicola Gigli and Sunra Mosconi.
\newblock The abstract {L}ewy-{S}tampacchia inequality and applications.
\newblock {\em J. Math. Pures Appl. (9)}, 104(2):258--275, 2015.

\bibitem{GunzburgerLehoucq}
Max Gunzburger and Richard~B. Lehoucq.
\newblock A nonlocal vector calculus with application to nonlocal boundary
  value problems.
\newblock {\em Multiscale Model. Simul.}, 8(5):1581--1598, 2010.

\bibitem{KassmannDeGiorgiNonlocal}
Moritz Kassmann.
\newblock The theory of {D}e {G}iorgi for non-local operators.
\newblock {\em C. R. Math. Acad. Sci. Paris}, 345(11):621--624, 2007.

\bibitem{KassmannHolder}
Moritz Kassmann.
\newblock A priori estimates for integro-differential operators with measurable
  kernels.
\newblock {\em Calc. Var. Partial Differential Equations}, 34(1):1--21, 2009.

\bibitem{KassmannHarnack}
Moritz Kassmann.
\newblock A new formulation of {H}arnack's inequality for nonlocal operators.
\newblock {\em C. R. Math. Acad. Sci. Paris}, 349(11-12):637--640, 2011.

\bibitem{KinderlehrerStampacchia}
David Kinderlehrer and Guido Stampacchia.
\newblock {\em An introduction to variational inequalities and their
  applications}, volume~31 of {\em Classics in Applied Mathematics}.
\newblock Society for Industrial and Applied Mathematics (SIAM), Philadelphia,
  PA, 2000.
\newblock Reprint of the 1980 original.

\bibitem{BoundedSolnEstimates}
Tommaso Leonori, Ireneo Peral, Ana Primo, and Fernando Soria.
\newblock Basic estimates for solutions of a class of nonlocal elliptic and
  parabolic equations.
\newblock {\em Discrete Contin. Dyn. Syst.}, 35(12):6031--6068, 2015.

\bibitem{MaRockner}
Zhi~Ming Ma and Michael R\"{o}ckner.
\newblock {\em Introduction to the theory of (nonsymmetric) {D}irichlet forms}.
\newblock Universitext. Springer-Verlag, Berlin, 1992.


\bibitem{Mosco1976}
Umberto Mosco.
\newblock Implicit variational problems and quasi variational inequalities.
\newblock In {\em Nonlinear operators and the calculus of variations ({S}ummer
  {S}chool, {U}niv. {L}ibre {B}ruxelles, {B}russels, 1975)}, pages 83--156.
  Lecture Notes in Math., Vol. 543. 1976.

\bibitem{MusinaNazarovVarIneqSpectral}
Roberta Musina and Alexander~I. Nazarov.
\newblock Variational inequalities for the spectral fractional {L}aplacian.
\newblock {\em Comput. Math. Math. Phys.}, 57(3):373--386, 2017.

\bibitem{MusinaNazarovSreenadhVarIneqFracLap}
Roberta Musina, Alexander~I. Nazarov, and Konijeti Sreenadh.
\newblock Variational inequalities for the fractional {L}aplacian.
\newblock {\em Potential Anal.}, 46(3):485--498, 2017.

\bibitem{Application}
Ricardo~H. Nochetto, Enrique Ot\'{a}rola, and Abner~J. Salgado.
\newblock Convergence rates for the classical, thin and fractional elliptic
  obstacle problems.
\newblock {\em Philos. Trans. Roy. Soc. A}, 373(2050):20140449, 14, 2015.

\bibitem{ObstacleProblems}
Jos\'{e}-Francisco Rodrigues.
\newblock {\em Obstacle problems in mathematical physics}, volume 134 of {\em
  North-Holland Mathematics Studies}.
\newblock North-Holland Publishing Co., Amsterdam, 1987.
\newblock Notas de Matem\'{a}tica [Mathematical Notes], 114.

\bibitem{RodriguesTeymurazyan}
Jos\'{e}~Francisco Rodrigues and Rafayel Teymurazyan.
\newblock On the two obstacles problem in {O}rlicz-{S}obolev spaces and
  applications.
\newblock {\em Complex Var. Elliptic Equ.}, 56(7-9):769--787, 2011.

\bibitem{RodriguesSantos}
José~Francisco Rodrigues and Lisa Santos.
\newblock On nonlocal variational and quasi-variational inequalities with
  fractional gradient.
\newblock {\em Applied Mathematics \& Optimization}, 80, no. 3:835 -- 852,
  2019 (see Correction in \url{https://arxiv.org/pdf/1903.02646.pdf}).

\bibitem{RosOton2016Survey}
Xavier Ros-Oton.
\newblock Nonlocal elliptic equations in bounded domains: a survey.
\newblock {\em Publ. Mat.}, 60(1):3--26, 2016.

\bibitem{RosOtonObsPb}
Xavier Ros-Oton.
\newblock Obstacle problems and free boundaries: an overview.
\newblock {\em SeMA J.}, 75(3):399--419, 2018.

\bibitem{RosOton2014Regularity}
Xavier Ros-Oton and Joaquim Serra.
\newblock The {D}irichlet problem for the fractional {L}aplacian: regularity up
  to the boundary.
\newblock {\em J. Math. Pures Appl. (9)}, 101(3):275--302, 2014.

\bibitem{SalsaSignorini}
Sandro Salsa.
\newblock Optimal regularity in lower dimensional obstacle problems.
\newblock In {\em Subelliptic {PDE}'s and applications to geometry and
  finance}, volume~6 of {\em Lect. Notes Semin. Interdiscip. Mat.}, pages
  217--226. Semin. Interdiscip. Mat. (S.I.M.), Potenza, 2007.
  
\bibitem{ServadeiValdinoci2013RMILewyStampacchia}
Raffaella Servadei and Enrico Valdinoci.
\newblock Lewy-{S}tampacchia type estimates for variational inequalities driven
  by (non)local operators.
\newblock {\em Rev. Mat. Iberoam.}, 29(3):1091--1126, 2013.

\bibitem{SS1}
Tien-Tsan Shieh and Daniel Spector.
\newblock On a new class of fractional partial differential equations.
\newblock {\em Advances in Calculus of Variations}, 8:321 -- 366, 2014.

\bibitem{SS2}
Tien-Tsan Shieh and Daniel Spector.
\newblock On a new class of fractional partial differential equations ii.
\newblock {\em Advances in Calculus of Variations}, 11:289 -- 307, 2017.

\bibitem{Silhavy}
Miroslav Silhavy.
\newblock Fractional vector analysis based on invariance requirements (critique
  of coordinate approaches).
\newblock {\em Continuum Mechanics and Thermodynamics}, 32, Issue 1:207 -- 288,
  2020.

\bibitem{SilvestrePaper}
Luis Silvestre.
\newblock Regularity of the obstacle problem for a fractional power of the
  {L}aplace operator.
\newblock {\em Comm. Pure Appl. Math.}, 60(1):67--112, 2007.

\bibitem{SilvestreThesis}
Luis~Enrique Silvestre.
\newblock {\em Regularity of the obstacle problem for a fractional power of the
  {L}aplace operator}.
\newblock ProQuest LLC, Ann Arbor, MI, 2005.
\newblock Thesis (Ph.D.)--The University of Texas at Austin.

\bibitem{Stampacchia1965}
Guido Stampacchia.
\newblock Le probl\`eme de {D}irichlet pour les \'{e}quations elliptiques du
  second ordre \`a coefficients discontinus.
\newblock {\em Ann. Inst. Fourier (Grenoble)}, 15(fasc. 1):189--258, 1965.


\bibitem{WarmaFracCap}
Mahamadi Warma.
\newblock The fractional relative capacity and the fractional {L}aplacian with
  {N}eumann and {R}obin boundary conditions on open sets.
\newblock {\em Potential Anal.}, 42(2):499--547, 2015.

\end{thebibliography}

\end{document}